\documentclass[11pt,a4paper,twoside]{amsart}
\usepackage[top=1in, bottom=1.25in, marginparwidth=0.8in, inner=1in, outer=1in]{geometry}
\usepackage{lmodern}
\usepackage[T2A,T1]{fontenc} 
\usepackage[utf8]{inputenc}
\usepackage[UKenglish]{babel}
\usepackage{yfonts}
\usepackage{etoolbox}
\usepackage{pinlabel}

\makeatletter
\newcommand*{\math@version@bold}{bold}
\DeclareMathOperator\DD{
  \textrm{%
    \usefont{T2A}{cmr}{\ifx\math@version\math@version@bold bx\else m\fi}{n}%
    \CYRD
  }%
} 
\makeatother

\usepackage{float}
\usepackage{graphicx}
\usepackage{xcolor}

\colorlet{myblue}{black}
\colorlet{gold}{yellow!90!black!70!red}
\colorlet{darkgreen}{green!70!black}
\colorlet{lightred}{red!30!white}
\colorlet{lightblue}{blue!30!white}

\usepackage{rotating}
\usepackage{tikz}
\usepackage{tikz-cd}  
\usetikzlibrary{arrows,arrows.meta}
\tikzcdset{arrow style=tikz, diagrams={>=latex}}
\usetikzlibrary{shapes.multipart}
\tikzset{>=latex}
\tikzset{
    labl/.style={anchor=south, rotate=90, inner sep=.5mm}
}

\usepackage{caption}
\usepackage[labelfont=up,margin=5pt]{subcaption}
\usepackage{wrapfig}
\newcommand\mystretch{\baselinestretch}
\newcommand{\myfixwrapfig}{\textcolor{white}{~}\vspace{-\mystretch\baselineskip\relax}}
\usepackage[section]{placeins}

\usepackage{setspace} 
\usepackage[bookmarks=true,pdfencoding=auto,bookmarksdepth=3]{hyperref}
\usepackage{bookmark}
\hypersetup{colorlinks=true,citecolor=black,urlcolor=myblue,linkcolor=black}

\usepackage{amsmath}
\usepackage{amsfonts}
\usepackage{amssymb}
\usepackage{amsthm}
\usepackage{mathtools}
\usepackage{mathrsfs}
\usepackage{stmaryrd}
\mathtoolsset{mathic=true}

\def\co{\colon\thinspace\relax}

\newtheorem{theorem}{Theorem}
\newtheorem*{theorem*}{Theorem}
\newtheorem*{main}{Main Theorem}
\newtheorem{lemma}[theorem]{Lemma}

\newtheorem{corollary}[theorem]{Corollary}
\newtheorem{proposition}[theorem]{Proposition}
\theoremstyle{definition}
\newtheorem{definition}[theorem]{Definition}
\newtheorem{question}[theorem]{Question}

\newtheorem{observation}[theorem]{Observation}
\newtheorem{remark}[theorem]{Remark}

\usepackage[makeroom]{cancel}
\usepackage{enumitem}
\usepackage[normalem]{ulem} 

\newcommand{\vc}[1]{\vcenter{\hbox{#1}}}%
\newcommand{\mypic}[3]{%
  \newcommand{#3}{%
    \vc{%
      \includegraphics[page=#2]%
        {PSTricks/KhCurves-pics-#1-pics.pdf}%
      }%
  }%
}%

\newcommand{\littlemor}[1]{%
  \left(%
    \vc{%
      \tikzcdset{diagrams={nodes={inner sep=2pt}}}%
      \tikzstyle{litARR}=[thick]%
			\tikzstyle{litarr}=[dotted]%
      \tikzstyle{litlab}=[fill=white,inner sep=0pt,rounded corners]%
      \begin{tikzcd}[column sep=1.25cm,row sep=0.5cm,ampersand replacement = \&, labels=description]
      #1
    \end{tikzcd}%
    }%
  \right) 
}

\mypic{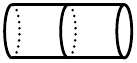}{1}{\tube}
\mypic{Cobordisms}{2}{\Torus}
\mypic{Cobordisms}{3}{\plane}
\mypic{Cobordisms}{4}{\planedot}

\mypic{Cobordisms}{5}{\planedotdot}
\mypic{Cobordisms}{6}{\Sphere}
\mypic{Cobordisms}{7}{\Spheredot}
\mypic{Cobordisms}{8}{\Spheredotdot}
\mypic{Cobordisms}{9}{\Spheredotdotdot}
\mypic{Cobordisms}{10}{\DiscT}
\mypic{Cobordisms}{11}{\DiscB}
\mypic{Cobordisms}{12}{\DiscR}
\mypic{Cobordisms}{13}{\DiscL}
\mypic{Cobordisms}{14}{\DiscTdot}
\mypic{Cobordisms}{15}{\DiscBdot}
\mypic{Cobordisms}{16}{\DiscRdot}
\mypic{Cobordisms}{17}{\DiscLdot}

\mypic{Cobordisms}{18}{\PairsOfPantsL}
\mypic{Cobordisms}{19}{\PairsOfPantsR}
\mypic{Cobordisms}{20}{\TuB}
\mypic{Cobordisms}{20,angle=180,origin=c}{\TuT}
\mypic{Cobordisms}{20,angle=-90,origin=c}{\TuL}
\mypic{Cobordisms}{20,angle=90,origin=c}{\TuR}
\mypic{Cobordisms}{21,angle=0,origin=c}{\TuDotA}
\mypic{Cobordisms}{21,angle=90,origin=c}{\TuDotB}
\mypic{Cobordisms}{21,angle=180,origin=c}{\TuDotC}
\mypic{Cobordisms}{21,angle=270,origin=c}{\TuDotD}
\mypic{Cobordisms}{22}{\TuNoDot}
	
\mypic{Cobordisms}{23}{\TorusRelI}
\mypic{Cobordisms}{24}{\TorusRelII}
\mypic{Cobordisms}{25}{\TorusRelIII}
\mypic{Cobordisms}{26}{\TorusRelIV}
\mypic{Cobordisms}{27}{\Star}
\mypic{Cobordisms}{28}{\StarH}
\mypic{Cobordisms}{29}{\StarDot}
\mypic{Cobordisms}{30}{\StarSheet}
\mypic{Cobordisms}{31}{\StarSheetH}
\mypic{Cobordisms}{32}{\StarTube}
\mypic{Cobordisms}{33}{\StarTubeTube}
\mypic{Cobordisms}{34}{\StarDotSameComp}
\mypic{Cobordisms}{35}{\StarNeck}
\mypic{Cobordisms}{36}{\StarNeckI}
\mypic{Cobordisms}{37}{\StarNeckII}
\mypic{Cobordisms}{38}{\StarNeckIII}
\mypic{Cobordisms}{39}{\StarNeckIp}
\mypic{Cobordisms}{40}{\StarNeckIIp}
\mypic{Cobordisms}{41}{\StarNeckIIIp}
\mypic{Cobordisms}{42}{\NoStarSheetH}
\mypic{Cobordisms}{43}{\NoStarSheetHRev}
\mypic{Cobordisms}{44}{\NoStarTube}
\mypic{Cobordisms}{45}{\StarDotRev}
\mypic{Cobordisms}{46}{\PairsOfPantsLAst}
\mypic{Cobordisms}{47}{\PairsOfPantsRAst}
\mypic{Cobordisms}{48}{\DiscRAst}
\mypic{Cobordisms}{49}{\PairsOfPantsGluedAst}
\mypic{Cobordisms}{50}{\planedotstar}
\mypic{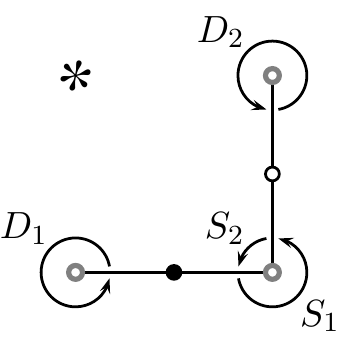}{1}{\FukayaFigureEightAlg}
\mypic{Curves}{2}{\FukayaHHHKAlg}

\mypic{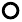}{1}{\DotC}
\mypic{Inline}{2}{\DotB}
\mypic{Inline}{3}{\DotCblue}
\mypic{Inline}{4}{\DotBblue}
\mypic{Inline}{5}{\DotCred}
\mypic{Inline}{6}{\DotBred}
\mypic{Inline}{1,scale=0.7}{\smallDotC}
\mypic{Inline}{2,scale=0.7}{\smallDotB}
\mypic{Inline}{3,scale=0.7}{\smallDotCblue}
\mypic{Inline}{4,scale=0.7}{\smallDotBblue}
\mypic{Inline}{5,scale=0.7}{\smallDotCred}
\mypic{Inline}{6,scale=0.7}{\smallDotBred}

\mypic{Inline}{7}{\Circle}
\mypic{Inline}{8}{\resolution}

\mypic{Inline}{9}{\Li}
\mypic{Inline}{10}{\Lo}
\mypic{Inline}{9,scale=0.7}{\smallLi}
\mypic{Inline}{10,scale=0.7}{\smallLo}
\mypic{Inline}{10,angle=270,origin=c}{\Lialt}
\mypic{Inline}{11}{\LiRed}
\mypic{Inline}{12}{\LoRed}
\mypic{Inline}{13}{\LiBlue}
\mypic{Inline}{14}{\LoBlue}
\mypic{Inline}{15}{\Ni}
\mypic{Inline}{16}{\No}
\mypic{Inline}{15,scale=0.7}{\smallNi}
\mypic{Inline}{16,scale=0.7}{\smallNo}
\mypic{Inline}{17}{\Lil}
\mypic{Inline}{18}{\Lol}
\mypic{Inline}{19}{\Nil}
\mypic{Inline}{20}{\Nol}
\mypic{Inline}{21}{\CrossingL}
\mypic{Inline}{22}{\CrossingR}
\mypic{Inline}{23}{\CrossingLDot}
\mypic{Inline}{24}{\CrossingRDot}
\mypic{Inline}{25}{\CrossingPos}
\mypic{Inline}{26}{\CrossingNeg}
\mypic{Inline}{27}{\CrossingPosMarkedi}
\mypic{Inline}{28}{\CrossingNegMarkedi}
\mypic{Inline}{29}{\CrossingPosMarkedii}
\mypic{Inline}{30}{\CrossingNegMarkedii}
\mypic{Inline}{31}{\CrossingPosMarkediii}
\mypic{Inline}{32}{\CrossingNegMarkediii}
\mypic{Inline}{33}{\CrossingPosMarkediv}
\mypic{Inline}{34}{\CrossingNegMarkediv}
\mypic{Inline}{25,angle=270,origin=c}{\CrossingPosR}
\mypic{Inline}{26,angle=270,origin=c}{\CrossingNegR}
\mypic{Inline}{27,angle=270,origin=c}{\CrossingPosMarkediR}
\mypic{Inline}{28,angle=270,origin=c}{\CrossingNegMarkediR}
\mypic{Inline}{29,angle=270,origin=c}{\CrossingPosMarkediiR}
\mypic{Inline}{30,angle=270,origin=c}{\CrossingNegMarkediiR}
\mypic{Inline}{31,angle=270,origin=c}{\CrossingPosMarkediiiR}
\mypic{Inline}{32,angle=270,origin=c}{\CrossingNegMarkediiiR}
\mypic{Inline}{33,angle=270,origin=c}{\CrossingPosMarkedivR}
\mypic{Inline}{34,angle=270,origin=c}{\CrossingNegMarkedivR}

\mypic{Inline}{35}{\LiO}
\mypic{Inline}{36}{\LiOl}
\mypic{Inline}{37}{\LilO}
\mypic{Inline}{38}{\LiOr}
\mypic{Inline}{39}{\LirO}
\mypic{Inline}{40}{\NiO}
\mypic{Inline}{41}{\NiOl}
\mypic{Inline}{42}{\NilO}
\mypic{Inline}{43}{\NiOr}
\mypic{Inline}{44}{\NirO}
\mypic{Inline}{45}{\LiDotL}
\mypic{Inline}{46}{\LiDotR}
\mypic{Inline}{47}{\LoDotT}
\mypic{Inline}{48}{\LoDotB}
\mypic{Inline}{49}{\NiDotL}
\mypic{Inline}{50}{\NiDotR}
\mypic{Inline}{51}{\NoDotT}
\mypic{Inline}{52}{\NoDotB}
\mypic{Inline}{53}{\TwistTwo}
\mypic{Inline}{53,angle=90,origin=c}{\TwistTwoUp}
\mypic{Inline}{54}{\TwistNegTwo}
\mypic{Inline}{55}{\TwistThree}
\mypic{Inline}{56}{\TwistNegThree}
\mypic{Inline}{57}{\TwistTwoDot}
\mypic{Inline}{58}{\TwistNegTwoDot}
\mypic{Inline}{59}{\TwistThreeDot}
\mypic{Inline}{60}{\TwistNegThreeDot}

\mypic{Inline}{61}{\CrossingLBasepoint}
\mypic{Inline}{61,angle=180,origin=c}{\CrossingLBasepointRotated}
\mypic{Inline}{62}{\CrossingRBasepoint}
\mypic{Inline}{62,angle=180,origin=c}{\CrossingRBasepointRotated}

\mypic{Inline}{63}{\LiT}
\mypic{Inline}{64}{\LoT}
\mypic{Inline}{65}{\NiT}
\mypic{Inline}{66}{\NoT}

\mypic{Inline}{67}{\singleB}
\mypic{Inline}{68}{\singleW}

\mypic{Inline}{69}{\ThreeTwistTangle}
\mypic{Inline}{70}{\ThreeTwistTangleBlue}

\mypic{Inline}{71}{\DotCarc}
\mypic{Inline}{72}{\DotBarc}
\mypic{Inline}{73}{\TrivialTwoTangle}
\mypic{Inline}{74}{\CircleDot}

\DeclareMathOperator{\Cob}{Cob}
\DeclareMathOperator{\Cobl}{\Cob_{/l}}
\DeclareMathOperator{\Cobb}{\Cob_{\bullet/l}}
\DeclareMathOperator{\Cobx}{\Cob_{\bullet/l}^{\bullet\ast=0}}

\DeclareMathOperator{\CoblD}{\Cobl(\Lo\oplus\Li)}
\DeclareMathOperator{\CobbD}{\Cobb(\Lo\oplus\Li)}
\DeclareMathOperator{\CobxD}{\Cobx(\Lo\oplus\Li)}

\DeclareMathOperator{\IdBimodBNAlg}{\prescript{}{\BNAlgH}{\mathbb{I}}^{\BNAlgH}}

\newcommand{\KhT}[1]{[\![ #1 ]\!]} 
\newcommand{\KhTl}[1]{[\![ #1 ]\!]_{/\text{l}}} 

\newcommand{\KhTbx}[1]{\KhT{#1}_{\bullet/l}^{\bullet\ast=0}} 

\DeclareMathOperator{\Kh}{Kh}
\newcommand{\Khr}{\widetilde{\Kh}}

\newcommand{\CF}{\operatorname{CF}}

\DeclareMathOperator{\Homology}{H_\ast}

\newcommand{\HHHKAlg}{\mathcal{A}}
\newcommand{\HHHKSubAlg}{\mathcal{A}^s}
\newcommand{\FigureEightAlg}{\mathcal{C}}
\newcommand{\FigureEightSubAlg}{\mathcal{C}^s}
\DeclareMathOperator{\BNAlg}{\mathscr{B}}
\DeclareMathOperator{\BNAlgH}{\mathcal{B}} 

\DeclareMathOperator{\id}{id} 
\DeclareMathOperator{\Mod}{Mod}

\newcommand{\mcIbim}{
{}_{\BNAlgH}[\mathbf{I}\rightarrow \mathbf{I}]^{\BNAlgH}}
\newcommand{\mcIbimSimple}{[\mathbf{I}\rightarrow \mathbf{I}]} 
\newcommand{\mcImap}{[\mathbf{I}\rightarrow \mathbf{I}]}

\newcommand{\QbimSimple}{\mathbf{Q}} 
\newcommand{\Qmap}{\mathbf{Q}} 
\newcommand{\Qbim}{{}_{\BNAlgH}\mathbf{Q}^{\BNAlg}} 
\newcommand{\YbimSimple}{\mathbf{Y}} 
\newcommand{\Ymap}{\mathbf{Y}} 
\newcommand{\Ybim}{{}_{\BNAlg}\mathbf{Y}^{\BNAlgH}}
\newcommand{\Fmap}{\mathbf{F}}

\newcommand{\F}{\mathbb{F}}

\newcommand{\fieldTwoElements}{\mathbb{F}}

\newcommand{\ignoreme}[1]{}

\newcommand{\bt}{\boxtimes}

\newcommand{\FPS}{(S^2, 4\pt)}
\newcommand{\wrFuk}{\mathcal{W}\FPS}

\newcommand{\genAlg}{\mathcal X}
\newcommand{\genAlgSecond}{\mathcal Y}

\newcommand{\gene}[1]{\text{#1}}
\newcommand{\Lk}{\mathcal{L}} 

\newcommand{\pt}{{\text{pt}}}

\begin{document}
\title[The tangle invariants agree]{Khovanov invariants via Fukaya categories:\\ the tangle invariants agree}

\author{Artem Kotelskiy}
\address{Department of Mathematics \\ Indiana University}
\email{artofkot@iu.edu}

\author{Liam Watson}
\address{Department of Mathematics \\ University of British Columbia}
\email{liam@math.ubc.ca}
\thanks{AK is supported by an AMS-Simons travel grant. LW is supported by an NSERC discovery/accelerator grant.}

\author{Claudius Zibrowius}
\address{Department of Mathematics \\ University of British Columbia}
\email{claudius.zibrowius@posteo.net}

\begin{abstract}
Given a pointed 4-ended tangle $T \subset D^3$, there are two Khovanov theoretic tangle invariants, $\DD_1(T)$ from~\cite{KWZ} and $L_T$ from~\cite{HHHK}, which are twisted complexes over the Fukaya category of the boundary 4-punctured sphere $\FPS=\partial (D^3, T)$. We prove that these two invariants are the same.
\end{abstract}

%

\maketitle

\section{Introduction}

A framework for studying Khovanov and Bar-Natan homologies via wrapped Lagrangian Floer theory of the 4-punctured sphere was recently developed~\cite{KWZ}. Given a 4-ended tangle decomposition of a link $\Lk$ along a Conway sphere $\FPS$, we defined various types of immersed curves inside the decomposing sphere $\FPS$, such that their wrapped Lagrangian Floer homology recovers Khovanov and Bar-Natan homologies of $\Lk$. These constructions drew inspiration from a paper of Hedden, Herald, Hogancamp, and Kirk \cite{HHHK}, and our aim here is to pin down the precise connection between the two works. 

A \emph{pointed} 4-ended tangle is a 4-ended tangle inside a $3$-ball $D^3$, one of whose tangle ends is distinguished from the other three. 
We usually choose the top left tangle end and mark it with an asterisk $\ast$, for example like so: $\CrossingLDot$. 
Let $\FPS$ denote the 4-punctured sphere which is the boundary of the 3-ball $D^3$ minus the four tangle ends, and write $\wrFuk$ for the wrapped Fukaya category of this 4-punctured sphere. 

Hedden, Herald, Hogancamp, and Kirk defined an invariant $L_T$ of pointed 4-ended tangles $T$, which takes the form of a twisted complex over a certain full subcategory $\HHHKAlg$ of $\wrFuk$~\cite{HHHK}; see Definition~\ref{def:HHHKAlg}. We denote the $A_\infty$-category of twisted complexes over $\HHHKAlg$ by $\Mod^{\HHHKAlg}$.  In previous work, we defined an invariant $\DD_1(T)$ of pointed 4-ended tangles $T$, which takes the form of a twisted complex over 
another full subcategory $\BNAlgH$ of $\wrFuk$~\cite{KWZ}; see Definition~\ref{def:BNAlgH}. We denote the $A_\infty$-category of twisted complexes over $\BNAlgH$ by $\Mod^{\BNAlgH}$. 
Both invariants are well-defined up to homotopy. 
 
\begin{main}
  There exists a quasi-isomorphism between \(\Mod^{\HHHKAlg}\) and a full subcategory of \(\Mod^{\BNAlgH}\) such that the image of \(L_T\) under the embedding 
  \[
  \Mod^{\HHHKAlg}\hookrightarrow\Mod^{\BNAlgH}
  \]
  is chain homotopic to \(\DD_1(T)\) for any pointed 4-ended tangle \(T\). In particular, as twisted complexes up to homotopy, \(L_T\) and \(\DD_1(T)\) determine each other. 
\end{main}

To be precise, the authors of~\cite{HHHK} actually work in the Fukaya category of the pillowcase instead of $\FPS$. The pillowcase is the traceless $SU(2)$-character variety of the 4-punctured sphere. As an orbifold, it is a $2$-sphere with four orbifold points. These orbifold points are treated as punctures in~\cite{HHHK}, and so there is no difference between working with the pillowcase or a 4-punctured sphere. Moreover, the pillowcase can be canonically identified with $\FPS$ after distinguishing one of the tangle ends~\cite[Section~8.2]{KWZ}. So in the discussion above, we are implicitly using the induced identification of the Fukaya category of the pillowcase with $\wrFuk$.

\subsection{Immersed curves}\label{sec:intro:curves} The invariant $\DD_1(T)$ can be represented geometrically by a collection of immersed curves with local systems~\cite{KWZ}. The authors of~\cite{HHHK} mention that it should be possible to reinterpret their invariant $L_T$ in terms of immersed curves with local systems by appealing to a general result by Haiden, Katzarkov, and Kontsevich~\cite[Theorem~4.3]{HKK}. Doing this in practice, however, is rather involved, because one needs to rewrite the twisted complex $L_T$ in terms of generators of $\wrFuk$ specified by some full arc system of the 4-punctured sphere. 

As a consequence of the Main Theorem, combined with results from~\cite{KWZ}, we now obtain a concrete way of producing immersed curves with local systems from $L_T$; compare~\cite[Corollary~1.3]{HHHK} and remarks following. Namely, we first use the equivalence of $L_T$ with $\DD_1(T)$, and then interpret $\DD_1(T)$ as the immersed curve invariant $\Khr(T)$, one of the main objects introduced in~\cite{KWZ}. An algorithm for computing $\Khr(T)$ has been implemented in~\cite{KhT.py}. 

\subsection{Notation, conventions, references.}

This paper is written as a follow-up to~\cite{KWZ}, and so we follow the same notation and conventions. All algebras in this paper are considered to be $A_\infty$-algebras, even in the case when the differential and the higher products vanish (as is the case, for example, for the algebra $\BNAlgH$). Consequently, all the maps between algebras are required to satisfy the corresponding relations. Moreover, all categories will be enriched over vector spaces. Therefore, we will use algebras and categories synonymously, treating every basic idempotent of an algebra as an object in its corresponding category. For derived structures over algebras we use the bordered algebraic framework of Lipshitz, Ozsváth, and Thurston~\cite{LOT-main,LOT-bim}. The main algebraic object for us is a \emph{right type~D structure} over an $A_\infty$-algebra ${\genAlg}$; see~\cite[Definition~2.2.23]{LOT-bim}. 
The notion of (bounded) type~D structures is equivalent to the notion of twisted complexes~\cite{Kontsevich_HMS}. This equivalence is described explicitly in \cite[Proposition~2.13]{KWZ} for type~D structures and twisted complexes over dg algebras, but the general case is similar. Following~\cite{LOT-bim}, we use the notation $\Mod^{\genAlg}$ for the $A_\infty$-category of type~D structures over ${\genAlg}$ that are chain homotopic to a bounded type~D structure, and use the superscript $N^{\genAlg}$ to indicate that $N\in \Mod^{\genAlg}$. In this paper all the type~D structures are bounded, since they arise from the cube of resolution construction where all the maps increase the cube filtration.

Type~D structures are preferred over twisted complexes because we often find it more convenient to analyze $A_\infty$-functors, eg $\mathbf{G}\co {\genAlg} \rightarrow {\genAlgSecond} $ or $\mathbf{G}\co {\genAlg} \rightarrow \Mod^{{\genAlgSecond}} $, in terms of the corresponding \emph{type~AD structures (bimodules)} $_{{\genAlg}}\mathbf{G}^{{\genAlgSecond} }$. See ~\cite[Definition~2.2.48]{LOT-bim} for the construction  in the case $\mathbf{G}\co {\genAlg} \rightarrow {\genAlgSecond}$; the case $\mathbf{G}\co {\genAlg} \rightarrow \Mod^{{\genAlgSecond}} $ is completely analogous. The induced map on the $A_\infty$-categories of type~D structures is given by the \emph{box tensor product}: $- \boxtimes {}_{{\genAlg}}\mathbf{G}^{{\genAlgSecond} }\co \Mod^{{\genAlg}} \rightarrow \Mod^{{\genAlgSecond}}$. Note that superscripts indicate the algebra from the type~D side of type~AD bimodules, whereas the subscript indicates the algebra from the type~A side. For the generators of bimodules, we use the subscripts to indicate their left and right idempotents. 

We work exclusively over the field~$\fieldTwoElements$ of two elements. 
  
\subsection{Outline}

Section~\ref{sec:algebras} introduces the algebras that are needed to define the tangle invariants $\DD_1(T)$ and $L_T$ in Section~\ref{sec:tangle_invariants}. Section~\ref{sec:embedding} constructs an embedding of $\HHHKAlg$ into $\Mod^{\BNAlgH}$ that is used in Section~\ref{sec:proof_of_main_theorem} to prove the Main Theorem. 

\section{Various algebras}\label{sec:algebras}
There are a number of algebras that  are relevant to the construction of the tangle invariants $L_T$ and \(\DD_1(T)\). The relationship between these algebras is summarized in Figure~\ref{fig:diagram:Algebras:woFigureEightAlg}.

\begin{figure}[ht]
  $$
  \begin{tikzcd}[column sep=1.5cm,row sep=1.2cm]
  \CoblD
  \arrow[equal]{r}{\sim}
  \arrow{d}[swap]{H=0}
  \arrow[d, phantom, ""{name=A}]
  &
  \BNAlgH
  \arrow{d}[swap]{H=0}
  \arrow[d, phantom, ""{name=B}]
  \arrow[phantom, from=A, to=B, "\square"]
  \\
  \CobxD
  \arrow[equal]{r}{\sim}
  &
  \BNAlg
  \arrow[equal]{r}{\sim}
  &
  \HHHKSubAlg
  \arrow[hook]{r}{}
  &
  \HHHKAlg
  \end{tikzcd}
  $$
  \caption{The relationship between the six algebras defined in Section~\ref{sec:algebras}. The square $\square$ commutes. }\label{fig:diagram:Algebras:woFigureEightAlg}
\end{figure}

\subsection{\texorpdfstring{The algebras \(\HHHKAlg\) and \(\HHHKSubAlg\)}{The algebras A and A\^\ s}}\label{sec:algebras:HHHKAlg}

\begin{definition}\label{def:HHHKAlg}
  Let $\HHHKAlg$ be the full subcategory of the Fukaya category $\wrFuk$ of the 4-punctured sphere generated by the two figure-8 curves $L_0$ and $L_1$ in Figure~\ref{fig:L0L1}. 
\end{definition}
It is a feature of the construction of the Fukaya category that $\HHHKAlg$ is well-defined only up to quasi-isomorphism. We now describe a particular model of $\HHHKAlg$ computed in~\cite{HHHK}. As vector spaces over $\fieldTwoElements$, the morphism spaces $\CF(L_i,L_j)$ in $\HHHKAlg$ are generated by the intersection points from Figure~\ref{fig:embedding}: 
\begin{align*}
\CF(L_0,L_0)
&=
\fieldTwoElements\langle a_0,b_0,c_0,d_0\rangle
&
\CF(L_1,L_1)
&=
\fieldTwoElements\langle a_1,b_1,c_1,d_1\rangle
\\
\CF(L_1,L_0)
&=
\fieldTwoElements\langle p_{01},q_{01}\rangle
&
\CF(L_0,L_1)
&=
\fieldTwoElements\langle p_{10},q_{10}\rangle
\end{align*}
\begin{figure}[H]
  \centering
  \begin{subfigure}[t]{0.3\textwidth}
    \centering
    \includegraphics[scale=0.5]{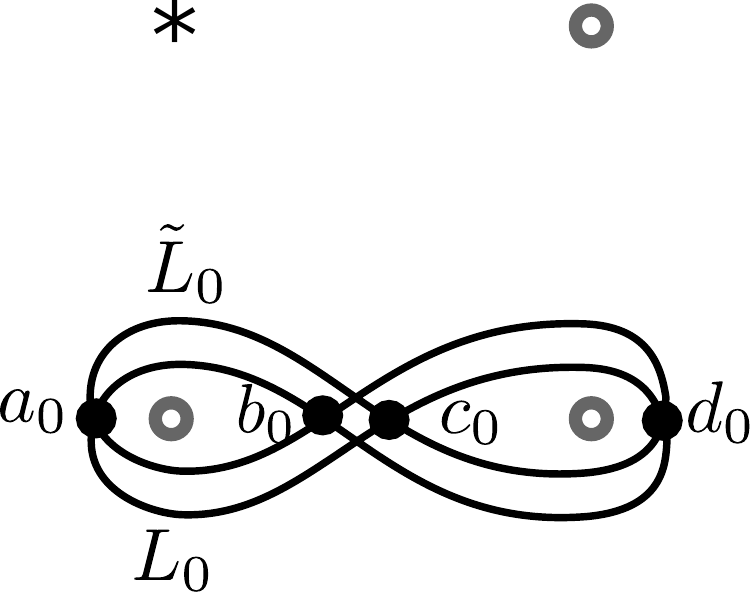}
    \caption{}  
  \end{subfigure}
  \begin{subfigure}[t]{0.3\textwidth}
    \centering
    \includegraphics[scale=0.5]{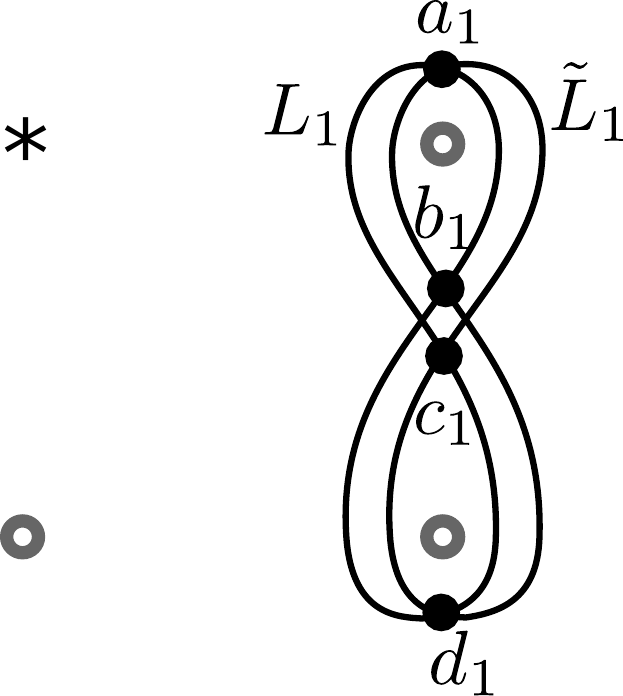}
    \caption{}
  \end{subfigure}
  \begin{subfigure}[t]{0.3\textwidth}
    \centering
    \includegraphics[scale=0.5]{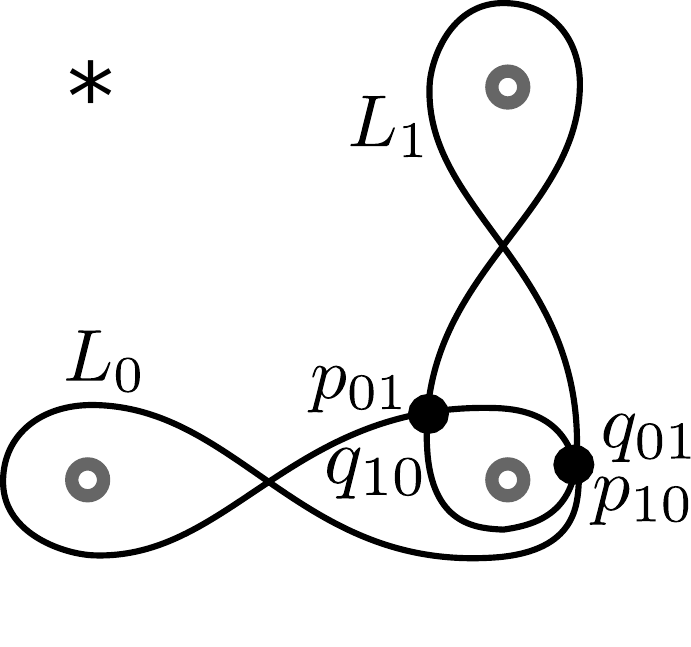}
    \caption{}
    \label{fig:L0L1}
  \end{subfigure}
  \caption{The generators of the algebra $\HHHKAlg$}\label{fig:embedding}
\end{figure}
We are using the same naming conventions for the morphisms as in~\cite{HHHK}, except that we have switched the roles of $p_{10}$ and $q_{10}$. The $A_\infty$-operations in $\HHHKAlg$ can also be described explicitly:

\begin{theorem}[{\cite[Theorem~4.1]{HHHK}}]\label{thm:HHHK:Algebra}
  The  $A_\infty$-category $\HHHKAlg$ is associative and strictly unital, with identity elements $a_0\in \CF(L_0,L_0)$ and $a_1\in \CF(L_1,L_1)$, ie $\mu^2(a_i, x)=x$ and $\mu^2(y,a_i)=y$ whenever these products are defined.  The only other non-zero values of $\mu^2$ are
  \begin{align*}
  \mu^2(b_0,c_0)=\mu^2(c_0,b_0)=\mu^2(p_{01},q_{10})=\mu^2(q_{01},p_{10})=d_0,~\mu^2(p_{01},p_{10})=c_0,\\
  ~\mu^2(b_1,c_1)=\mu^2(c_1,b_1)=\mu^2(p_{10},q_{01})=\mu^2(q_{10},p_{01})=d_1,~\mu^2(p_{10},p_{01})=c_1\\
  \mu^2(b_0,p_{01})=\mu^2(p_{01}, b_1 )=q_{01}, ~ ~\mu^2(p_{10}, b_0 )=\mu^2(b_1,p_{10})=q_{10}.
  \end{align*}  
  The non-zero values of $\mu^3$ are
  \begin{align*}
  \mu^3(p_{01}, p_{10}, b_0)=a_0,~\mu^3(p_{10},b_0,p_{01})=a_1, \\
  \mu^3 (q_{01}, p_{10}, b_0)=\mu^3(p_{01},q_{10},b_0)=b_0, 
  ~\mu^3(q_{10},p_{01}, b_1)= \mu^3(b_1,p_{10},q_{01})=b_1, \\
  \mu^3(c_0,b_0,c_0)= \mu^3(c_0,q_{01},p_{10})= \mu^3(c_0,p_{01},q_{10})=c_0,\\
  \mu^3(c_1,b_1,c_1)= \mu^3(c_1,q_{10},p_{01})=\mu^3(p_{10},q_{01}, c_1)=c_1,\\
  \mu^3(d_0,c_0,b_0)=\mu^3(b_0,c_0,d_0)=\mu^3(q_{01},p_{10},d_0)= \mu^3(p_{01},q_{10},d_0)=d_0,\\
  \mu^3(b_1,c_1,d_1)=\mu^3(d_1,c_1,b_1)=\mu^3(d_1,p_{10},q_{01})=\mu^3(q_{10}, p_{01}, d_1)=d_1,\\
  ~ \mu^3(p_{01},q_{10},q_{01})=q_{01},~ \mu^3(q_{10},p_{01},q_{10})=q_{10},~
  \mu^3(p_{10},q_{01},p_{10})=\mu^3(p_{10},p_{01},q_{10})=p_{10}.
  \end{align*}
  Furthermore, $\mu^k$ is zero if $k\ne 2,3$.
\end{theorem}


\begin{remark}\label{rem:asymmetry:HHHK}
  The $A_\infty$-operations of $\HHHKAlg$ are asymmetrical. This is an artifact of the particular choice of representatives of the curves $L_0$ and $L_1$ and the Hamiltonian function used for perturbations; see the bottom of~\cite[Page~10]{HHHK}. There exists a different Hamiltonian function which results in a symmetric model of the quasi-isomorphism class of the algebra $\HHHKAlg$.
\end{remark}

\begin{definition}\label{def:HHHKSubAlg}
  Let $\HHHKSubAlg$ be the subalgebra of $\HHHKAlg$ generated by the morphisms \(a_0\), \(c_0\), \(a_1\), \(c_1\), \(p_{01}\) and \(p_{10}\).
\end{definition}
Observe that the subalgebra $\HHHKSubAlg$ has no higher $A_\infty$-multiplications;  it is an honest algebra over $\fieldTwoElements$. In~\cite{HHHK}, this algebra is denoted by $A_{\mathcal D}$.

\begin{wrapfigure}{r}{0.22\textwidth}
  \centering
  $\FukayaFigureEightAlg$
  \caption{}\label{fig:DotBarcDotCarc}
\end{wrapfigure}

\myfixwrapfig

\subsection{\texorpdfstring{The algebras \(\BNAlgH\) and \(\BNAlg\)}{The algebras B and curly B}}\label{sec:algebras:BNAlgH}

\begin{definition}\label{def:BNAlgH}
  Let $\BNAlgH$ be the full subcategory of the wrapped Fukaya category \(\wrFuk\) of the 4-punctured sphere
  generated by the two arcs labelled \(\DotB\) and \(\DotC\) in Figure~\ref{fig:DotBarcDotCarc}. 
\end{definition}

It is well-known (see \cite[Theorem~7.6]{Bocklandt}, for example) that the wrapped Fukaya category generated by arcs on a surface admits an explicit description in terms of a quiver algebra. In our case of the category $\BNAlgH$, we have
\[
\BNAlgH =
\F\Big[
\begin{tikzcd}[row sep=2cm, column sep=1.5cm]
\DotB
\arrow[leftarrow,in=145, out=-145,looseness=5]{rl}[description]{D_1}
\arrow[leftarrow,bend left]{r}[description]{S_2}
&
\DotC
\arrow[leftarrow,bend left]{l}[description]{S_1}
\arrow[leftarrow,in=35, out=-35,looseness=5]{rl}[description]{D_2}
\end{tikzcd}
\Big]\Big/ (D_j S_i=0=S_i D_j)
\hspace{0.22\textwidth}
\]
Often, we will abuse notation by using $D$ in place of both $D_1$ and $D_2$, and $S$ in place of both $S_1$ and $S_2$. More precisely, set $S\coloneqq S_1+S_2$ and $D\coloneqq D_1+D_2$.
The variable $H$ denotes the following central element of $\BNAlgH$:
  $$H\coloneqq D+S^2 = D_1 + D_2 + S_2S_1 + S_1S_2$$

\begin{definition}\label{def:BNAlg}
  Let \(\BNAlg=\BNAlgH|_{H=0}\) be the quotient algebra obtained from \(\BNAlgH\) by setting \(H=0\). 
\end{definition}

The algebra $\BNAlg$ admits the following description as a quiver algebra:
\[
\BNAlg
=
\F\Big[
\begin{tikzcd}[row sep=2cm, column sep=1.5cm,ampersand replacement = \&]
\DotB
\arrow[leftarrow,in=145, out=-145,looseness=5]{rl}[description]{D}
\arrow[leftarrow,bend left]{r}[description]{S}
\&
\DotC
\arrow[leftarrow,bend left]{l}[description]{S}
\arrow[leftarrow,in=35, out=-35,looseness=5]{rl}[description]{D}
\end{tikzcd}
\Big]
\Big/ 
\Big(\substack{DS=0=SD\\S^2=D}\Big)
=
\F\Big[
\begin{tikzcd}[row sep=2cm, column sep=1.5cm,ampersand replacement = \&]
\DotB
\arrow[leftarrow,bend left]{r}[description]{S}
\&
\DotC
\arrow[leftarrow,bend left]{l}[description]{S}
\end{tikzcd}
\Big]
\Big/ (S^3=0) 
\]

\subsection{\texorpdfstring{The cobordism categories $\Cobl$ and $\Cobb$}{The cobordism categories Cob\_/\text{l} and Cob\_•/\text{l}}}\label{sec:algebras:Cob}

\begin{definition}\label{def:Cobl}
  Let \( \Cobl\) be the following category:
  the objects are crossingless tangle diagrams \( T \) in a fixed disc \(D^2\) with four ends on \(\partial D^2\). The morphisms are $\fieldTwoElements$-linear combinations of orientable cobordisms between those tangle diagrams modulo the following relations:
  \begin{description}
    \item[\( S \)-relation] Whenever a cobordism contains a component which is a sphere, the cobordism is set equal to 0; in short: \[ \Sphere=0\]
    \item[\( 4Tu \)-relation] Given a cobordism \( C \), let us consider four embedded open discs \( D_1 \) through \( D_4 \) on \( C \). 
    For \( i\neq j \), let \( C_{ij} \) denote the cobordism obtained from \( C \) by removing the discs \( D_i \) and \( D_j \) and replacing them by a tube with the same boundary. Then \[ C_{12}+C_{34}=C_{13}+C_{24}\] 
    or in pictures: 
    \[ \TuT+\TuB=\TuR +\TuL\]
  \end{description} 
\end{definition}

\begin{definition}\label{def:CoblDot}
  Let \(\Cobb\) be the following category:
  its objects are the same as in \(\Cobl\). 
  Morphisms are $\fieldTwoElements$-linear combinations of orientable \emph{dotted} cobordisms. The components of such cobordisms carry extra decorations in the form of marked points, which we label by dots~$\bullet$; they can move freely along each component. The morphisms are considered modulo the following set of relations:
  \begin{description}
    \item[\( S \)-relation] Whenever a cobordism contains a component which is a sphere, the cobordism is set equal to 0; in short: \[ \Sphere=0\]
    \item[\(S_\bullet\)-relation] Whenever a cobordism contains a component which is a sphere with a single dot~\(\bullet\), it is equal to the same cobordism but with this component removed: \[ \Spheredot=1\]
    \item[$\bullet\bullet$-relation] Whenever a cobordism contains a component containing two dots occupying the same component of a cobordism, the cobordism is set equal to 0: \[ \planedotdot=0\]
    \item[Neck-cutting relation] Given a cobordism \(C\) containing a compressing disc \(D\), consider the cobordism \(C'\) obtained from \(C\) by doing surgery along \(D\), which contains two embedded discs \(D_1\) and \(D_2\) as mementos from the surgery. Then the morphism represented by \(C\) is equal to the formal sum of two morphisms which are obtained from \(C'\) by placing a dot in \(D_1\) and \(D_2\), respectively. In pictures: \[\tube=\DiscLdot\DiscR+\DiscL\DiscRdot\] 
  \end{description}
\end{definition}

Both categories $\Cobl$ and $\Cobb$ were introduced by Bar-Natan~\cite{BarNatanKhT}, and in~\cite[Remark~4.12]{KWZ} we explain the connection between them. The main idea is to choose a particular basis for the morphism spaces in $\Cobl$. This is done as follows: first, we fix a special component for every cobordism in \( \Cobl \). To make these choices consistently, we distinguish one of the four tangle ends (the same for all tangles, usually the one on the top left), and say that the special component of any cobordism between two crossingless tangles is the one containing this distinguished tangle end. Note that this is compatible with composition of cobordisms. In local pictures of cobordisms, we will decorate the special component by an asterisk \( \ast \). Similarly, we mark the distinguished tangle end by an asterisk $\ast$, like so: $\Lo$. Then we introduce two pieces of specific notation. First, we introduce a formal variable $H$; multiplication of a cobordism by $H$ increases the genus of the special component by 1. Second, we introduce decorations of components of cobordisms by dots $\bullet$, which  represent certain linear combinations of cobordisms in $\Cobl$. For details, see~\cite[Definition~4.10]{KWZ}. With this notation in place,~\cite[Proposition~4.11]{KWZ} and~\cite[Proposition~4.15]{KWZ} imply that the two relations on morphisms in $\Cobl$ are equivalent to 
\begin{equation}
\begin{split}
&\Sphere=0 \qquad \Spheredot=1 \qquad \planedotdot=H\cdot\planedot \qquad \planedotstar =0 \\ 
&\tube=\DiscLdot\DiscR+\DiscL\DiscRdot+H\cdot\DiscL\DiscR 
\end{split}
\end{equation}
By comparing these relations to the relations in Definition~\ref{def:CoblDot}, we obtain the following relationship between $\Cobl$ and  $\Cobb$:
\begin{equation}\label{fig:relationship_Cobs}
\Cobl \Big/\Big( H=0 \Big)
=
\Cobb \Big/\left(\planedotstar =0\right)
\end{equation}
\begin{definition}
Define \(\Cobx\) to be the category from Equation~\eqref{fig:relationship_Cobs}.
\end{definition}

\begin{remark}\label{rem:delooping}
An important feature of the categories $\Cobl$,  $\Cobb$, and \(\Cobx\) is that every object is isomorphic to a direct sum of the basic crossingless tangles. This results from a process referred to as \emph{delooping}; see~\cite[Lemma~3.1]{BarNatanBurgosSoto} for $\Cobb$ and~\cite[Observation~4.18]{KWZ} for $\Cobl$. 
Specializing to the case of 4-ended tangles, there are exactly two basic crossingless tangles: $\Lo$ and $\Li$. Delooping implies that $\Cobl$ is equivalent to the additive enlargement of its full subcategory generated by $\Lo$ and $\Li$, and the same holds for $\Cobb$ and $\Cobx$. We denote those subcategories by
$$
\CoblD, \quad \CobbD, \quad\text{ and }\quad \CobxD,
$$
respectively. 
\end{remark}

\subsection{Various coincidences}\label{sec:algebras:relations}

In~\cite[Theorem~4.21]{KWZ}, we showed the following:

\begin{theorem}\label{thm:identification:BNAlgH}
 There is an isomorphism $\BNAlgH\cong\CoblD$ via the following dictionary:
  \begin{align*}
    \DotB
    &
    \leftrightarrow \Lo,
    &
    (\DotB \xrightarrow{\iota_{\smallDotB}} \DotB)
    &
    \leftrightarrow \id_{\smallLo},
    &
    (\DotB \xrightarrow{D} \DotB) 
    &
    \leftrightarrow \textnormal{(dot cobordism)},
    &
    (\DotC \xrightarrow{S} \DotB) 
    &
    \leftrightarrow \textnormal{(saddle cobordism)}
    \\
    \DotC
    &
    \leftrightarrow \Li,
    &
    (\DotC \xrightarrow{\iota_{\smallDotC}} \DotC)
    & 
    \leftrightarrow \id_{\smallLi}, 
    &
    (\DotC \xrightarrow{D} \DotC) 
    &
    \leftrightarrow \textnormal{(dot cobordism)}, 
    &
    (\DotB \xrightarrow{S} \DotC) 
    &
    \leftrightarrow \textnormal{(saddle cobordism)}
  \end{align*}
\end{theorem}
By dot cobordism, we mean the identity cobordism with a dot on the component without the distinguished tangle end \(\ast\), and by saddle cobordism we mean the cobordism whose underlying surface is contractible and does not carry any dot. Note that since we are using the dot notation in this theorem, the isomorphism depends on the choice of the distinguished tangle end $\ast$. 

\begin{observation}\label{obs:identification:BNAlg}
  The isomorphism from Theorem~\ref{thm:identification:BNAlgH} restricts to $\BNAlg\cong\CobxD$.
\end{observation}

One of the key results in~\cite{HHHK}, namely Theorem~7.1, states that Bar-Natan's dotted cobordism category of planar crossingless tangles is isomorphic to the additive enlargement of the subalgebra $\HHHKSubAlg$. In view of Observation~\ref{obs:identification:BNAlg}, the following result is simply a reformulation of theirs: 
\begin{theorem}
There is an isomorphism  $\BNAlg\cong\HHHKSubAlg$ via the following dictionary:
  \begin{align*}
    \DotB
    &
    \leftrightarrow L_0,
    &
    (\DotB \xrightarrow{\iota_{\smallDotB}} \DotB)
    &
    \leftrightarrow a_0,
    &
    (\DotB \xrightarrow{S^2} \DotB) 
    &
    \leftrightarrow c_0,
    &
    (\DotC \xrightarrow{S} \DotB) 
    &
    \leftrightarrow p_{01}
    \\
    \DotC
    &
    \leftrightarrow L_1,
    &
    (\DotC \xrightarrow{\iota_{\smallDotC}} \DotC)
    & 
    \leftrightarrow a_1, 
    &
    (\DotC \xrightarrow{S^2} \DotC) 
    &
    \leftrightarrow c_1, 
    &
    (\DotB \xrightarrow{S} \DotC) 
    &
    \leftrightarrow p_{10}
  \end{align*}
\end{theorem}

\section{Complex-valued tangle invariants}\label{sec:tangle_invariants}

We recall the definitions of the invariants \(\DD_1(T)\) and \(L_T\) associated with a pointed 4-ended tangle \(T\). A schematic picture is shown in Figure~\ref{fig:diagram:Mod:woFigureEightAlg}, which we will explain in more detail throughout this section. To give just a rough overview, in the center of this diagram we see various categories of complexes over the algebras from Section~\ref{sec:algebras}. The relationship between those categories is informed by the relationship between the corresponding algebras summarized in Figure~\ref{fig:diagram:Algebras:woFigureEightAlg}. 
The top and bottom of the diagram show the intermediate steps in the definitions of $\DD_1(T)$ and $L_T$. 

\begin{figure}[ht]
  \centering
  $
  \begin{tikzcd}[column sep=1.5cm,row sep=0.4cm]
  &
  \KhTl{T}
  \arrow[mapsto]{r}
  \arrow[d, phantom, "\rotatebox{-90}{$\in$}"]
  &
  \DD(T)^{\BNAlgH}
  \arrow[mapsto]{r}
  \arrow[d, phantom, "\rotatebox{-90}{$\in$}"]
  &
  \DD_1(T)
  \arrow[d, phantom, "\rotatebox{-90}{$\in$}"]
  \\
  &
  \Mod^{\Cobl}
  \arrow[equal]{r}{\sim}
  \arrow{dd}[swap]{H=0}
  \arrow[dd, phantom, ""{name=A}]
  &
  \Mod^{\BNAlgH}
  \arrow{r}{-\bt \mcIbim}
  \arrow{dd}[swap]{H=0}
  \arrow[dd, phantom, ""{name=B}]
  &
  \Mod^{\BNAlgH}
  \arrow[phantom, from=A, to=B, "\square"]
  \\
  T
  \arrow[in=180,out=45,mapsto]{uur}
  \arrow[in=180,out=-45,mapsto]{ddr}
  \arrow[phantom, to=A, "\triangleleft"]
  \\
  &
  \Mod^{\Cobx}
  \arrow[equal]{r}{\sim}
  &
  \Mod^{\BNAlg}
  \arrow[equal]{r}{\sim}
  &
  \Mod^{\HHHKSubAlg}
  \arrow[hook]{r}{}
  &
  \Mod^{\HHHKAlg}
  \\
  &
  \KhTbx{T}
  \arrow[mapsto]{r}
  \arrow[bend right=20,mapsto]{rrr}[name=E, swap]{\text{\protect\cite{HHHK}}}
  \arrow[u, phantom, "\rotatebox{90}{$\in$}"]
  &
  \DD(T)^{\BNAlg}
  \arrow[mapsto]{r}[name=F]{}
  \arrow[u, phantom, "\rotatebox{90}{$\in$}"]
  &
  \DD(T)^{\HHHKSubAlg}
  \arrow[mapsto]{r}
  \arrow[u, phantom, "\rotatebox{90}{$\in$}"]
  &
  L_T
  \arrow[u, phantom, "\rotatebox{90}{$\in$}"]
  \arrow[phantom, from=E, to=F, "\diamond"]
  \end{tikzcd}
  $
  \caption{A schematic overview of the construction of the tangle invariants $\DD_1(T)$ and $L_T$. The triangle $\triangleleft$ and the squares $\square$ and $\diamond$ commute. }\label{fig:diagram:Mod:woFigureEightAlg}
\end{figure}
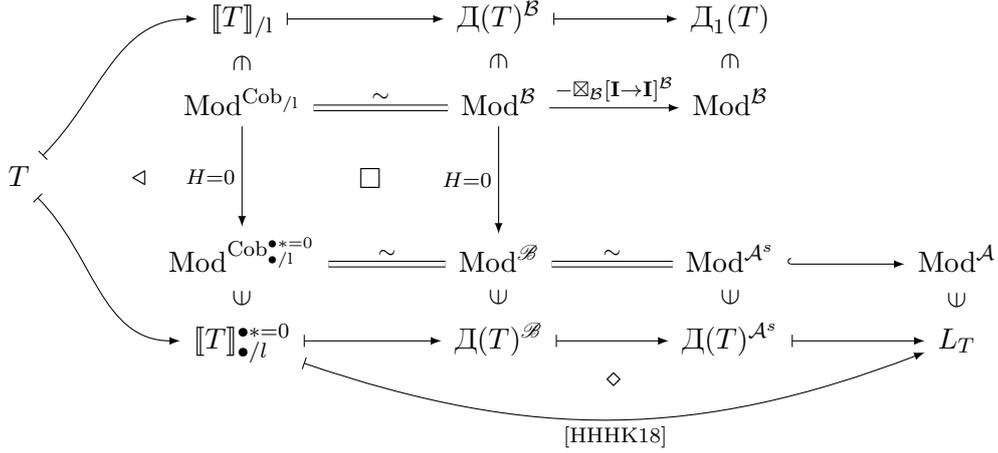

\subsection{\texorpdfstring{The invariant \(\DD_1(T)\)}{The invariant Д₁(T)}}

The definition of \(\DD_1(T)\) starts with Bar-Natan's invariant \( \KhTl{T}\) of oriented tangles \(T\) in the 3-ball~\cite{BarNatanKhT}. This invariant generalizes Khovanov homology of links in the 3-sphere. 
Like Khovanov's link invariant, \(\KhTl{T}\) takes the form of a bigraded chain complex  which is well-defined up to chain homotopy equivalence. 
However, while the link invariant is a chain complex over $\F$, \(\KhTl{T}\) is a chain complex over the cobordism category \(\Cobl\) from Section~\ref{sec:algebras:Cob}.  

The next step relies in an essential way on the fact that we work with pointed 4-ended tangles. First, we deloop the chain complex \(\KhTl{T}\), thus writing it as a chain complex over $\CoblD$, as explained in Remark~\ref{rem:delooping}. Then, using the isomorphism from Theorem~\ref{thm:identification:BNAlgH} between $\CoblD$ and the algebra $\BNAlgH$ specified by the distinguished endpoint $\ast$ of $T$, we obtain the type~D structure $\DD(T)^{\BNAlgH}$. We highlight that the choice of $\ast$ simply specifies a particular basis of the morphism spaces inside $\Cobl$.

In the last step, we take the mapping cone of the $H$ times the identity map on $\DD(T)^{\BNAlgH}$:
$$
\DD_1(T)
\coloneqq 
[\DD(T)^{\BNAlgH} \xrightarrow{H\cdot \id} \DD(T)^{\BNAlgH}]
$$
Equivalently, we can view \(\DD_1(T)\) as the result of taking the box tensor product of $\DD(T)^{\BNAlgH}$ with the type~AD bimodule $\mcIbim$ defined by
$$
\mcIbim \coloneqq [\IdBimodBNAlg  \xrightarrow{(-|H)}  \IdBimodBNAlg]
$$
where $\IdBimodBNAlg$ denotes the identity bimodule of the algebra $\BNAlgH$. 

\subsection{\texorpdfstring{The invariant $L_T$}{The invariant L\_T}}

The first step in the definition of $L_T$ is very similar to that of $\DD_1(T)$. The authors of~\cite{HHHK} also start with some version of Bar-Natan's tangle invariant over a cobordism category. However, they work over $\Cobx$ instead of $\Cobl$. We denote the corresponding tangle invariant by $\KhTbx{T}$. Then they observe that the algebra $\CobxD$ is isomorphic to the subalgebra $\HHHKSubAlg$ of $\HHHKAlg$ \cite[Proposition~6.3]{HHHK}. This allows them to construct a functor from $\Cobx$ to the additive enlargement of $\HHHKAlg$ \cite[Theorem~7.1]{HHHK}. This functor induces a functor from $\Mod^{\Cobx}$ to $\Mod^{\HHHKAlg}$, and the image of $\KhTbx{T}$ under this functor is $L_T$. 

\subsection{A preliminary comparison of the two invariants}

It its clear from the construction that the image of $\KhTl{T}$ under the functor induced by the quotient map $\Cobl\rightarrow\Cobx$ is equal to $\KhTbx{T}$, explaining the commutativity of the triangle labelled $\triangleleft$ on the left of Figure~\ref{fig:diagram:Mod:woFigureEightAlg}.

Furthermore, the single step from $\KhTbx{T}$ to $L_T$ can be broken into three steps, as shown at the bottom of that diagram. Namely, we may first deloop $\KhTbx{T}$, ie write it as a chain complex over the full subcategory $\CobxD$ and then use the isomorphism from Observation~\ref{obs:identification:BNAlg} to obtain a complex $\DD(T)^{\BNAlg}$ over $\BNAlg$. The commutativity of the square $\square$ in Figure~\ref{fig:diagram:Algebras:woFigureEightAlg} implies the commutativity of the square labelled $\square$ in Figure~\ref{fig:diagram:Mod:woFigureEightAlg}. Then, we may use the isomorphism from Theorem~\ref{thm:identification:BNAlgH} to obtain a complex $\DD(T)^{\HHHKSubAlg}$ over $\HHHKSubAlg$. We may regard $\DD(T)^{\HHHKSubAlg}$ as a complex over $\HHHKAlg$ via the embedding $\Mod^{\HHHKSubAlg}\hookrightarrow\Mod^{\HHHKAlg}$. We claim that this complex agrees with $L_T$.

Indeed, the composition of the isomorphisms between $\CobxD$, $\BNAlg$ and $\HHHKSubAlg$ agrees with the isomorphism between $\CobxD$ and $\HHHKSubAlg$ used in~\cite{HHHK}. So the only difference between the two constructions is that the authors of~\cite{HHHK} do not use delooping, but work with chain complexes over the \emph{whole} category $\Cobx$ instead. To define the embedding of $\Cobx$ into the additive enlargement of $\HHHKAlg$, the authors describe all the morphisms in the category $\Cobx$, see~\cite[Proposition~6.5]{HHHK}. That proposition, in its content, is equivalent to delooping. So we see that the square labelled $\diamond$ in Figure~\ref{fig:diagram:Mod:woFigureEightAlg} also commutes. 

The main conclusion from the commutativity of the whole diagram in Figure~\ref{fig:diagram:Mod:woFigureEightAlg} is the following:

\begin{proposition}
The tangle invariants $\DD_1(T)$ and $L_T$ can be both recovered from the invariant $\DD(T)^{\BNAlgH}$. 
\end{proposition}

\subsection{Prospectus on the Main Theorem}
The strategy for the proof of the Main Theorem is as follows: we define yet another pair of categories, namely the full subcategory $\FigureEightAlg$ of $\Mod^{\BNAlgH}$ generated by the two mapping cones, and its subcategory $\FigureEightSubAlg$. We then construct a quasi-isomorphism between $\FigureEightAlg$ and $\HHHKAlg$ which restricts to a quasi-isomorphism between $\FigureEightSubAlg$ and $\HHHKSubAlg$. This allows us to view $L_T$ as an object in $\Mod^{\BNAlgH}$ that can be compared to $\DD_1(T)$; see Figures~\ref{fig:diagram:Algebras:overview} and~\ref{fig:diagram:Mod:overview}.

\section{\texorpdfstring{An embedding of $\HHHKAlg$ into $\Mod^{\BNAlgH}$}{An embedding of the algebra A into Mod\^\ B}}\label{sec:embedding}

We define a full subcategory $\FigureEightAlg$ of $\Mod^{\BNAlgH}$ and a subcategory $\FigureEightSubAlg$ of $\FigureEightAlg$ and show that they are quasi-isomorphic to $\HHHKAlg$ and $\HHHKSubAlg$, respectively. This allows us to embed $\HHHKAlg$ into $\Mod^{\BNAlgH}$. 
Figure~\ref{fig:diagram:Algebras:overview} shows how the categories $\FigureEightAlg$ and $\FigureEightSubAlg$ extend the diagram from Figure~\ref{fig:diagram:Algebras:woFigureEightAlg}. 

\begin{figure}[ht]
  $$
  \begin{tikzcd}[column sep=1.5cm,row sep=1.2cm]
  \CoblD
  \arrow[equal]{r}{\sim}
  \arrow{d}[swap]{H=0}
  \arrow[d, phantom, ""{name=A}]
  &
  \BNAlgH
  \arrow{r}{}
  \arrow[bend left=15]{rrr}{\mcImap}
  \arrow{d}[swap]{H=0}
  \arrow[d, phantom, ""{name=B}]
  &
  \FigureEightSubAlg
  \arrow[hook]{r}{}
  \arrow[d, phantom, ""{name=C}]
  &
  \FigureEightAlg
  \arrow[hook]{r}{}
  \arrow[d, phantom, ""{name=D}]
  &
  \Mod^{\BNAlgH}
  \arrow[phantom, from=A, to=B, "\square"]
  \arrow[phantom, from=B, to=C, "\boxdot"]
  \arrow[phantom, from=C, to=D, "\square"]
  \\
  \CobxD
  \arrow[equal]{r}{\sim}
  &
  \BNAlg
  \arrow[equal]{r}{\sim}
  &
  \HHHKSubAlg
  \arrow[hook]{r}{}
  \arrow{u}[left]{\rotatebox{90}{$\sim$}}[right]{\Fmap }
  &
  \HHHKAlg
  \arrow{u}[left]{\rotatebox{90}{$\sim$}}[right]{\Fmap }
  \end{tikzcd}
  $$
  \caption{An update of the commutative diagram from Figure~\ref{fig:diagram:Algebras:woFigureEightAlg} which now includes the algebras $\FigureEightAlg$ and $\FigureEightSubAlg$, the functor $\Fmap $ defined in Section~\ref{sec:embedding}, and the functor $\mcImap$ corresponding to the type~AD bimodule $\mcIbim$. In Section~\ref{sec:proof_of_main_theorem}, we will show that the square labelled $\boxdot$ commutes up to homotopy. }\label{fig:diagram:Algebras:overview}
\end{figure}

\begin{definition}
  Let $\FigureEightAlg$ be the full subcategory of $\Mod^{\BNAlgH}$ generated by the two objects 
  $$
  [\DotB \xrightarrow{H} \DotB]
  \quad
  \text{ and }
  \quad
  [\DotC \xrightarrow{H} \DotC]
  $$ 
\end{definition}

\begin{figure}[p]
  $$
  \begin{aligned}
  A^k_{0}\coloneqq&
  \littlemor{
    \DotB
    \arrow[litARR]{d}
    \arrow[litarr]{r}[litlab]{S^{2k}}
    \&
    \DotB
    \arrow[litARR]{d}
    \\
    \DotB
    \arrow[litarr]{r}[litlab]{S^{2k}}
    \&
    \DotB
  }
  &
  B^k_{0}\coloneqq&
  \littlemor{
    \DotB
    \arrow[litARR]{d}
    \arrow[litarr]{rd}[litlab]{S^{2k}}
    \&
    \DotB
    \arrow[litARR]{d}
    \\
    \DotB
    \&
    \DotB
  }
  &
  C^l_{0}\coloneqq&
  \littlemor{
    \DotB
    \arrow[litARR]{d}
    \arrow[litarr]{r}[litlab]{D^{l}}
    \&
    \DotB
    \arrow[litARR]{d}
    \\
    \DotB
    \arrow[litarr]{r}[litlab]{D^{l}}
    \&
    \DotB
  }
  &
  D^l_{0}\coloneqq&
  \littlemor{
    \DotB
    \arrow[litARR]{d}
    \arrow[litarr]{rd}[litlab]{D^{l}}
    \&
    \DotB
    \arrow[litARR]{d}
    \\
    \DotB
    \&
    \DotB
  }
  \\
  A^k_{1}\coloneqq&
  \littlemor{
    \DotC
    \arrow[litARR]{d}
    \arrow[litarr]{r}[litlab]{S^{2k}}
    \&
    \DotC
    \arrow[litARR]{d}
    \\
    \DotC
    \arrow[litarr]{r}[litlab]{S^{2k}}
    \&
    \DotC
  }
  &
  B^k_{1}\coloneqq&
  \littlemor{
    \DotC
    \arrow[litARR]{d}
    \arrow[litarr]{rd}[litlab]{S^{2k}}
    \&
    \DotC
    \arrow[litARR]{d}
    \\
    \DotC
    \&
    \DotC
  }
  &
  C^l_{1}\coloneqq&
  \littlemor{
    \DotC
    \arrow[litARR]{d}
    \arrow[litarr]{r}[litlab]{D^{l}}
    \&
    \DotC
    \arrow[litARR]{d}
    \\
    \DotC
    \arrow[litarr]{r}[litlab]{D^{l}}
    \&
    \DotC
  }
  &
  D^l_{1}\coloneqq&
  \littlemor{
    \DotC
    \arrow[litARR]{d}
    \arrow[litarr]{rd}[litlab]{D^{l}}
    \&
    \DotC
    \arrow[litARR]{d}
    \\
    \DotC
    \&
    \DotC
  }
  \\
  P^{l}_{10}\coloneqq&
  \littlemor{
    \DotB
    \arrow[litARR]{d}
    \arrow[litarr]{r}[litlab]{S^{2l-1}}
    \&
    \DotC
    \arrow[litARR]{d}
    \\
    \DotB
    \arrow[litarr]{r}[litlab]{S^{2l-1}}
    \&
    \DotC
  }
  &
  Q^l_{10}\coloneqq&
  \littlemor{
    \DotB
    \arrow[litARR]{d}
    \arrow[litarr]{rd}[litlab]{S^{2l-1}}
    \&
    \DotC
    \arrow[litARR]{d}
    \\
    \DotB
    \&
    \DotC
  }
  &
  P^{l}_{01}\coloneqq&
  \littlemor{
    \DotC
    \arrow[litARR]{d}
    \arrow[litarr]{r}[litlab]{S^{2l-1}}
    \&
    \DotB
    \arrow[litARR]{d}
    \\
    \DotC
    \arrow[litarr]{r}[litlab]{S^{2l-1}}
    \&
    \DotB
  } 
  &
  Q^l_{01}\coloneqq&
  \littlemor{
    \DotC
    \arrow[litARR]{d}
    \arrow[litarr]{rd}[litlab]{S^{2l-1}}
    \&
    \DotB
    \arrow[litARR]{d}
    \\
    \DotC
    \&
    \DotB
  } 
  \end{aligned}
  $$
  \caption{A basis of the kernel of the differential of the dg algebra $\FigureEightAlg$ as a vector space over $\fieldTwoElements$. The vertical arrows are labelled by $H=D+S^2$, and $k\geq0, l\geq1$.}\label{fig:FigureEightAlg:basisI}
\end{figure}

\begin{figure}[p]
    $$
    \begin{aligned}
    \hat{A}^k_{0}\coloneqq&
    \littlemor{
        \DotB
        \arrow[litARR]{d}
        \&
        \DotB
        \arrow[litARR]{d}
        \\
        \DotB
        \arrow[litarr]{ru}[litlab]{S^{2k}}
        \&
        \DotB
    }
    &
    \hat{B}^k_{0}\coloneqq&
    \littlemor{
        \DotB
        \arrow[litARR]{d}
        \&
        \DotB
        \arrow[litARR]{d}
        \\
        \DotB
        \arrow[litarr]{r}[litlab]{S^{2k}}
        \&
        \DotB
    }
    &
    \hat{C}^l_{0}\coloneqq&
    \littlemor{
        \DotB
        \arrow[litARR]{d}
        \&
        \DotB
        \arrow[litARR]{d}
        \\
        \DotB
        \arrow[litarr]{ru}[litlab]{D^{l}}
        \&
        \DotB
    }
    &
    \hat{D}^l_{0}\coloneqq&
    \littlemor{
        \DotB
        \arrow[litARR]{d}
        \&
        \DotB
        \arrow[litARR]{d}
        \\
        \DotB
        \arrow[litarr]{r}[litlab]{D^{l}}
        \&
        \DotB
    }\\
    \hat{A}^k_{1}\coloneqq&
    \littlemor{
        \DotC
        \arrow[litARR]{d}
        \&
        \DotC
        \arrow[litARR]{d}
        \\
        \DotC
        \arrow[litarr]{ru}[litlab]{S^{2k}}
        \&
        \DotC
    }
    &
    \hat{B}^k_{1}\coloneqq&
    \littlemor{
        \DotC
        \arrow[litARR]{d}
        \&
        \DotC
        \arrow[litARR]{d}
        \\
        \DotC
        \arrow[litarr]{r}[litlab]{S^{2k}}
        \&
        \DotC
    }
    &
    \hat{C}^l_{1}\coloneqq&
    \littlemor{
        \DotC
        \arrow[litARR]{d}
        \&
        \DotC
        \arrow[litARR]{d}
        \\
        \DotC
        \arrow[litarr]{ru}[litlab]{D^{l}}
        \&
        \DotC
    }
    &
    \hat{D}^l_{1}\coloneqq&
    \littlemor{
        \DotC
        \arrow[litARR]{d}
        \&
        \DotC
        \arrow[litARR]{d}
        \\
        \DotC
        \arrow[litarr]{r}[litlab]{D^{l}}
        \&
        \DotC
    }
    \\
    \hat{P}^l_{10}\coloneqq&
    \littlemor{
        \DotB
        \arrow[litARR]{d}
        \&
        \DotC
        \arrow[litARR]{d}
        \\
        \DotB
        \arrow[litarr]{ru}[litlab]{S^{2l-1}}
        \&
        \DotC
    }
    &
    \hat{Q}^l_{10}\coloneqq&
    \littlemor{
        \DotB
        \arrow[litARR]{d}
        \&
        \DotC
        \arrow[litARR]{d}
        \\
        \DotB
        \arrow[litarr]{r}[litlab]{S^{2l-1}}
        \&
        \DotC
    }
    &
    \hat{P}^l_{01}\coloneqq&
    \littlemor{
        \DotC
        \arrow[litARR]{d}
        \&
        \DotB
        \arrow[litARR]{d}
        \\
        \DotC
        \arrow[litarr]{ru}[litlab]{S^{2l-1}}
        \&
        \DotB
    }
    &
    \hat{Q}^l_{01}\coloneqq&
    \littlemor{
        \DotC
        \arrow[litARR]{d}
        \&
        \DotB
        \arrow[litARR]{d}
        \\
        \DotC
        \arrow[litarr]{r}[litlab]{S^{2l-1}}
        \&
        \DotB
    }
    \end{aligned}
    $$
    \caption{Families of morphisms of the dg algebra $\FigureEightAlg$ which complete those morphisms from Figure~\ref{fig:FigureEightAlg:basisI} to a basis of $\FigureEightAlg$ as a vector space over $\fieldTwoElements$. Notation is that of Figure~\ref{fig:FigureEightAlg:basisI}.}\label{fig:FigureEightAlg:basisII}
\end{figure}

\begin{figure}[p]
  \begin{align*}
  \Fmap ^2(p_{01},p_{10})&=
  \littlemor{
      \DotB
      \arrow[litARR]{d}
      \&
      \DotB
      \arrow[litARR]{d}
      \\
      \DotB
      \arrow[litarr]{ru}[litlab]{1}
      \&
      \DotB
  }=\hat{A}^0_0
  &
  \Fmap ^2(p_{10},p_{01})&=
  \littlemor{
      \DotC
      \arrow[litARR]{d}
      \&
      \DotC
      \arrow[litARR]{d}
      \\
      \DotC
      \arrow[litarr]{ru}[litlab]{1}
      \&
      \DotC
  }=\hat{A}^0_1
  \\
  \Fmap ^2(c_{0},c_{0})&=
  \littlemor{
      \DotB
      \arrow[litARR]{d}
      \&
      \DotB
      \arrow[litARR]{d}
      \\
      \DotB
      \arrow[litarr]{ru}[litlab]{D}
      \&
      \DotB
  }=\hat{C}^1_0
  &
  \Fmap ^2(c_{1},c_{1})&=
  \littlemor{
      \DotC
      \arrow[litARR]{d}
      \&
      \DotC
      \arrow[litARR]{d}
      \\
      \DotC
      \arrow[litarr]{ru}[litlab]{D}
      \&
      \DotC
  }=\hat{C}^1_1
  \\
  \Fmap ^2(q_{01},p_{10})&=
  \littlemor{
      \DotB
      \arrow[litARR]{d}
      \&
      \DotB
      \arrow[litARR]{d}
      \\
      \DotB
      \arrow[litarr]{r}[litlab]{1}
      \&
      \DotB
  }=\hat{B}^0_0
  &
  \Fmap ^2(q_{10},p_{01})&=
  \littlemor{
      \DotC
      \arrow[litARR]{d}
      \&
      \DotC
      \arrow[litARR]{d}
      \\
      \DotC
      \arrow[litarr]{r}[litlab]{1}
      \&
      \DotC
  }=\hat{B}^0_1
  \\
  \Fmap ^2(p_{01},q_{10})&=
  \littlemor{
      \DotB
      \arrow[litARR]{d}
      \&
      \DotB
      \arrow[litARR]{d}
      \\
      \DotB
      \arrow[litarr]{r}[litlab]{1}
      \&
      \DotB
  }=\hat{B}^0_0
  &
  \Fmap ^2(p_{10},q_{01})&=
  \littlemor{
      \DotC
      \arrow[litARR]{d}
      \arrow[litarr]{r}[litlab]{1}
      \&
      \DotC
      \arrow[litARR]{d}
      \\
      \DotC
      \&
      \DotC
  }=A^0_1+\hat{B}^0_1
  \\
  \Fmap ^2(d_0,c_0)&=
  \littlemor{
      \DotB
      \arrow[litARR]{d}
      \&
      \DotB
      \arrow[litARR]{d}
      \\
      \DotB
      \arrow[litarr]{r}[litlab]{D}
      \&
      \DotB
  }=\hat{D}^1_0
  &
  \Fmap ^2(d_1,c_1)&=
  \littlemor{
      \DotC
      \arrow[litARR]{d}
      \&
      \DotC
      \arrow[litARR]{d}
      \\
      \DotC
      \arrow[litarr]{r}[litlab]{D}
      \&
      \DotC
  }=\hat{D}^1_1
  \\
  \Fmap ^2(c_0,d_0)&=
  \littlemor{
      \DotB
      \arrow[litARR]{d}
      \arrow[litarr]{r}[litlab]{D}
      \&
      \DotB
      \arrow[litARR]{d}
      \\
      \DotB
      \&
      \DotB
  }=C^1_0+\hat{D}^1_0
  &
  \Fmap ^2(c_1,d_1)&=
  \littlemor{
      \DotC
      \arrow[litARR]{d}
      \arrow[litarr]{r}[litlab]{D}
      \&
      \DotC
      \arrow[litARR]{d}
      \\
      \DotC
      \&
      \DotC
  }=C^1_1+\hat{D}^1_1
  \\
  \Fmap ^3(c_0,d_0,c_0)&=
  \littlemor{
      \DotB
      \arrow[litARR]{d}
      \&
      \DotB
      \arrow[litARR]{d}
      \\
      \DotB
      \arrow[litarr]{ru}[litlab]{D}
      \&
      \DotB
  }=\hat{C}^1_0
  &
  \Fmap ^3(c_1,d_1,c_1)&=
  \littlemor{
      \DotC
      \arrow[litARR]{d}
      \&
      \DotC
      \arrow[litARR]{d}
      \\
      \DotC
      \arrow[litarr]{ru}[litlab]{D}
      \&
      \DotC
  }=\hat{C}^1_1
  \end{align*}
  \caption{The non-zero actions of $\Fmap ^2$ and $\Fmap ^3$. The asymmetry between $\Fmap ^2(p_{01},q_{10})$ and $\Fmap ^2(p_{10},q_{01})$ is forced by the asymmetry of the $A_\infty$-operations of $\HHHKAlg$, see Remark~\ref{rem:asymmetry:HHHK}.}\label{fig:Functor:2actions}
\end{figure}

A basis of the algebra $\FigureEightAlg$ as a vector space over $\fieldTwoElements$ is shown in Figures~\ref{fig:FigureEightAlg:basisI} and~\ref{fig:FigureEightAlg:basisII}. The multiplication and the differential is inherited from $\Mod^{\BNAlgH}$. $A^0_0$ and $A^0_1$ are the identity morphisms. The non-zero actions of the differential $\mu^1$ are given by
$$
\mu^1(\hat{X}^{k})=X^{k+1}
\quad\text{for all }k\geq1
\text{ and } 
X\in\{A_i,B_i,C_i,D_i,P_{j},Q_{j}\mid i\in\{0,1\}, j\in\{01,10\}\}$$
and
\begin{align*}
\mu^1(\hat{A}^0_{i})=A^1_{i}+C^1_{i} 
\quad\text{ and }\quad
\mu^1(\hat{B}^0_{i})=B^1_{i}+D^1_{i}
\quad\text{for } i\in\{0,1\}.
\end{align*}

\begin{definition}\label{fig:Functor}
    Let $\Fmap \co\HHHKAlg\rightarrow\FigureEightAlg$ be the $A_\infty$-functor defined as follows. On the level of objects, we set
    $$
    \Fmap (L_0)=[\DotB \xrightarrow{H} \DotB]
    \quad\text{ and }\quad
    \Fmap (L_1)=[\DotC \xrightarrow{H} \DotC].
    $$
    We then define
    $$
    \Fmap ^1(a_i)=A^0_i
    \quad
    \Fmap ^1(b_i)=B^0_i
    \quad
    \Fmap ^1(c_i)=C^1_i
    \quad
    \Fmap ^1(d_i)=D^1_i
    \quad
    \Fmap ^1(p_j)=P^1_j
    \quad
    \Fmap ^1(q_j)=Q^1_j
    $$
    for all $i\in\{0,1\}$ and $j\in\{01,10\}$. The non-zero actions of $\Fmap ^2$ and $\Fmap ^3$  are shown in Figure~\ref{fig:Functor:2actions}.
    Furthermore, $\Fmap ^k=0$ for $k\geq4$.
\end{definition}

\begin{lemma}
    \(\Fmap \) is a well-defined $A_\infty$-functor.
\end{lemma}

\begin{proof}
  We only need to check the $A_\infty$-relations on sequences of length at most six, because the multiplications $\mu^i_{\HHHKAlg}$ in $\HHHKAlg$ vanish for $i\neq2,3$, the multiplications $\mu^i_{\FigureEightAlg}$ in $\FigureEightAlg$ vanish for $i\neq1,2$ and $\Fmap ^j$ vanishes for $j\neq1,2,3$. Moreover, for sequences of length six, the only possible non-vanishing term in the $A_\infty$-relation is 
  $$
  \mu^2_{\FigureEightAlg}(\Fmap ^3(f,e,d),\Fmap ^3(c,b,a))\quad (\text{for }a,b,c,d,e,f\in\HHHKAlg)
  $$
  We can easily check by hand that this expression vanishes, since the only non-vanishing actions of $\Fmap ^3$ are $\Fmap ^3(c_0,d_0,c_0)$ and $\Fmap ^3(c_1,d_1,c_1)$ and they square to zero. Similarly, for sequences of length one, the $A_\infty$-relation boils down to checking if the images of the basis elements of $\HHHKAlg$ under $\Fmap ^1$ lie in the kernel of $\mu^1_{\FigureEightAlg}$. Checking the relations on the sequences of length strictly between one and six is more tedious. However, it is a finite calculation, so it can be checked by a computer; see the \texttt{C++} script \texttt{functor.cpp} \cite{functor.cpp}. 
\end{proof}

\begin{definition}
  Let \(\FigureEightSubAlg\) be the subalgebra of \(\FigureEightAlg\) generated by \(A_i^k\), \(\hat{A}_i^k\), \(C_i^l\), \(\hat{C}_i^l\), \(P_j^l\) and \(\hat{P}_j^l\) for all \(k\geq0\), \(l\geq1\), \(i\in\{0,1\}\) and \(j\in\{01,10\}\).
\end{definition}

\begin{proposition}\label{prop:quasi-iso}
  The functor \(\Fmap \) is a quasi-isomorphism between the two \(A_\infty\)-algebras \(\HHHKAlg\) and \(\FigureEightAlg\). Furthermore, \(\Fmap \) restricts to a quasi-isomorphism between the two subalgebras \(\HHHKSubAlg\) and \(\FigureEightSubAlg\).
\end{proposition}
\begin{proof}
  From the description of the differentials on $\FigureEightAlg$, it is clear that $\Fmap $ induces an isomorphism on homology. Indeed, $\Homology(\HHHKAlg) = \HHHKAlg $ and
  $$
  \Homology(\Fmap ) 
  \co 
  \HHHKAlg
  \rightarrow 
  \Homology(\FigureEightAlg)
  $$ 
  is an isomorphism since $[A^0_{i}]$, $[B^0_{i}]$, $[C^1_{i}]$, $[D^1_{i}]$, $[P^1_{j}]$ and $[Q^1_{j}]$ for $i\in\{0,1\}$ and $j\in\{01,10\}$ generate $\Homology(\FigureEightAlg)$. The second statement follows from the definitions.
\end{proof}

\section{Proof of the Main Theorem}
\label{sec:proof_of_main_theorem}

Figure~\ref{fig:diagram:Mod:overview} shows an updated version of Figure~\ref{fig:diagram:Mod:woFigureEightAlg} which now includes the quasi-isomorphism $\Fmap $ and the two algebras $\FigureEightAlg$ and $\FigureEightSubAlg$ from Section~\ref{sec:embedding}. 
Definition~\ref{def:HHHKSubAlg} of $\FigureEightSubAlg$ implies that the functor $-\bt \mcIbim$ factors through $\Mod^{\FigureEightSubAlg}$. 
We denote the corresponding tangle invariants in $\Mod^{\FigureEightSubAlg}$ and $\Mod^{\FigureEightAlg}$ by $\DD(T)^{\FigureEightSubAlg}$ and $\DD(T)^{\FigureEightAlg}$, respectively. 
If we can show that the square in the middle of the diagram in Figure~\ref{fig:diagram:Algebras:overview} labelled by $\boxdot$ commutes up to homotopy, then so does the square labelled $\boxdot$ in Figure~\ref{fig:diagram:Mod:overview}; this is the remaining obstacle on the way to the proof of the Main Theorem.

\begin{figure}[ht]
  $$
  \begin{tikzcd}[column sep=1.4cm,row sep=0.4cm]
  &
  \KhTl{T}
  \arrow[mapsto]{r}
  \arrow[d, phantom, "\rotatebox{-90}{$\in$}"]
  &
  \DD(T)
  \arrow[mapsto]{r}
  \arrow[d, phantom, "\rotatebox{-90}{$\in$}"]
  &
  \DD(T)^{\FigureEightSubAlg}
  \arrow[mapsto]{r}
  \arrow[d, phantom, "\rotatebox{-90}{$\in$}"]
  &
  \DD(T)^{\FigureEightAlg}
  \arrow[mapsto]{r}
  \arrow[d, phantom, "\rotatebox{-90}{$\in$}"]
  &
  \DD_1(T)
  \arrow[d, phantom, "\rotatebox{-90}{$\in$}"]
  \\
  &
  \Mod^{\Cobl}
  \arrow[equal]{r}{\sim}
  \arrow{dd}[swap]{H=0}
  \arrow[dd, phantom, ""{name=A}]
  &
  \Mod^{\BNAlgH}
  \arrow{r}{}
  \arrow{dd}[swap]{H=0}
  \arrow[dd, phantom, ""{name=B}]
  &
  \Mod^{\FigureEightSubAlg}
  \arrow[hook]{r}{}
  \arrow[dd, phantom, ""{name=C}]
  &
  \Mod^{\FigureEightAlg}
  \arrow[hook]{r}{}
  \arrow[dd, phantom, ""{name=D}]
  &
  \Mod^{\BNAlgH}
  \arrow[phantom, from=A, to=B, "\square"]
  \arrow[phantom, from=B, to=C, "\boxdot"]
  \arrow[phantom, from=C, to=D, "\square"]
  \\
  T
  \arrow[in=180,out=45,mapsto]{uur}
  \arrow[in=180,out=-45,mapsto]{ddr}
  \arrow[phantom, to=A, "\triangleleft"]
  \\
  &
  \Mod^{\Cobx}
  \arrow[equal]{r}{\sim}
  &
  \Mod^{\BNAlg}
  \arrow[equal]{r}{\sim}
  &
  \Mod^{\HHHKSubAlg}
  \arrow[hook]{r}{}
  \arrow{uu}[left]{\rotatebox{90}{$\sim$}}[right]{\Fmap }
  &
  \Mod^{\HHHKAlg}
  \arrow{uu}[left]{\rotatebox{90}{$\sim$}}[right]{\Fmap }
  \\
  &
  \KhTbx{T}
  \arrow[mapsto]{r}
  \arrow[u, phantom, "\rotatebox{90}{$\in$}"]
  &
  \DD(T)^{\BNAlg}
  \arrow[mapsto]{r}
  \arrow[u, phantom, "\rotatebox{90}{$\in$}"]
  &
  \DD(T)^{\HHHKSubAlg}
  \arrow[mapsto]{r}
  \arrow[u, phantom, "\rotatebox{90}{$\in$}"]
  &
  L_T
  \arrow[u, phantom, "\rotatebox{90}{$\in$}"]
  \end{tikzcd}
  $$
  \caption{An update of Figure~\ref{fig:diagram:Mod:woFigureEightAlg} taking Section~\ref{sec:embedding} into account }\label{fig:diagram:Mod:overview}
\end{figure}

Let us first consider the commutativity of the square $\boxdot$ from Figure~\ref{fig:diagram:Algebras:overview} after composition with the embeddings $\FigureEightSubAlg\hookrightarrow \FigureEightAlg \hookrightarrow \Mod^{\BNAlgH}$. Denote the composition of functors
$$
\BNAlg\cong\HHHKSubAlg\xrightarrow{\Fmap }\FigureEightSubAlg\hookrightarrow \FigureEightAlg\hookrightarrow \Mod^{\BNAlgH}
$$
by $\Ymap$. The induced map on type~D structures is the box tensor product with the type~AD bimodule $\Ybim$ from Figure~\ref{fig:bimodules:Y}. Denote the map 
$$\BNAlgH\xrightarrow{H=0} \BNAlg$$
by $\Qmap$ (‘‘quotienting''). The induced map $\Mod^{\BNAlgH}\xrightarrow{H=0}\Mod^{\BNAlg}$ amounts to box tensoring with the bimodule $\Qbim$ from Figure~\ref{fig:bimodules:Q}. 

\begin{lemma}\label{lem:main}
  The bimodules $\mcIbim$ and $ \Qbim \bt \Ybim$ are homotopy equivalent. In other words, the following diagram commutes up to homotopy:
  $$
  \begin{tikzcd}[column sep=1.7cm,row sep=1.2cm]
  \BNAlgH
  \arrow{d}[swap]{\Qmap}
  \arrow{rr}{\mcImap}
  &&
  \Mod^{\BNAlgH}
  \\
  \BNAlg
  \arrow{rru}[swap]{\Ymap}
  \end{tikzcd}
  $$
\end{lemma}

  \begin{figure}[p]
    \centering
    \begin{subfigure}{0.48\textwidth}
      $$
      \begin{tikzcd}[column sep=2.5cm,row sep=1.8cm,ampersand replacement = \&]
      {}_{\DotB}\gene{t}_{\DotB}
      \arrow[bend left=12]{r}[description]{(S|S)}
      \arrow[in=-150,out=150,looseness=5.5]{rl}[description]{(S^2|D)}
      \arrow[bend left=15, d, "\substack{(-|D) \\ (-|S^2)}" right]
      \&
      {}_{\DotC}\gene{u}_{\DotC}
      \arrow[bend left=12]{l}[description]{(S|S)}
      \arrow[in=-30,out=30,looseness=5.5]{rl}[description]{(S^2|D)}
      \arrow[bend right=15, d, "\substack{(-|D) \\ (-|S^2)}" left ]
      \\
      {}_{\DotB}\gene{k}_{\DotB}
      \arrow[bend left=12]{r}[description]{(S|S)}
      \arrow[in=-150,out=150,looseness=5.5]{rl}[description]{(S^2|D)}
      \arrow[bend left=15, u, "\substack{(S^2,S^2|D) \\ (S,S|1)}" left ]
      \&
      {}_{\DotC}\gene{v}_{\DotC}
      \arrow[bend left=12]{l}[description]{(S|S)}
      \arrow[in=-30,out=30,looseness=5.5]{rl}[description]{(S^2|D)}
      \arrow[bend right=15, u, "\substack{(S^2,S^2|D) \\ (S,S|1)}" right ]
      \end{tikzcd}
      $$
      \caption{$\Ybim$}\label{fig:bimodules:Y}
    \end{subfigure}
    \begin{subfigure}{0.48\textwidth}
      \vspace{20pt}
      $$
      \begin{tikzcd}[column sep=2.4cm, ampersand replacement = \&]
      {}_{\DotB} \gene{z}_{\DotB}
      \arrow[leftarrow,in=145, out=-145,looseness=8]{rl}[description]{\substack{(D|S^2) \\ (S^2|S^2)}}
      \arrow[leftarrow,bend left]{r}[description]{(S|S)}
      \&
      {}_{\DotC} \gene{w}_{\DotC}
      \arrow[leftarrow,bend left]{l}[description]{(S|S)}
      \arrow[leftarrow,in=35, out=-35,looseness=8]{rl}[description]{\substack{(D|S^2) \\ (S^2|S^2)}}
      \end{tikzcd}
      $$
      \vspace{20pt}
      \caption{$\Qbim$}
      \label{fig:bimodules:Q}
    \end{subfigure}
    \\
  \begin{subfigure}{0.48\textwidth}
    \centering
    $$
    \begin{tikzcd}[column sep=2.4cm,row sep=2.4cm]
    {}_{\DotB}\gene{l}_{\DotB}
    \arrow[bend left=12]{r}[above]{(S^{2k+1}|S^{2k+1})}
    \arrow[in=-150,out=150,looseness=5.5]{rl}[above, near start]{\substack{(D^{k+1}|D^{k+1}) \\ (S^{2k+2}|S^{2k+2})}}
    \arrow[d, "\substack{(-|D) \\ (-|S^2)}" description]
    &
    {}_{\DotC} \gene{b}_{\DotC}
    \arrow[bend left=12]{l}[below]{(S^{2k+1}|S^{2k+1})}
    \arrow[in=-30,out=30,looseness=5.5]{rl}[above, very near start]{\substack{(D^{k+1}|D^{k+1}) \\ (S^{2k+2}|S^{2k+2})}}
    \arrow[d, "\substack{(-|D) \\ (-|S^2)}" description]
    \\
    {}_{\DotB}\gene{m}_{\DotB}
    \arrow[bend left=12]{r}[above]{(S^{2k+1}|S^{2k+1})}
    \arrow[in=-150,out=150,looseness=5.5]{rl}[below, near end ]{\substack{(D^{k+1}|D^{k+1}) \\ (S^{2k+2}|S^{2k+2})}}
    &
    {}_{\DotC}\gene{y}_{\DotC}
    \arrow[bend left=12]{l}[below]{(S^{2k+1}|S^{2k+1})}
    \arrow[in=-30,out=30,looseness=5.5]{rl}[below, very near end]{\substack{(D^{k+1}|D^{k+1}) \\ (S^{2k+2}|S^{2k+2})}}
    \end{tikzcd}
    $$
    \caption{$\mcIbim$}
    \label{fig:bimodules:I}
  \end{subfigure}
  \begin{subfigure}{0.48\textwidth}
    \centering
    $$
    \begin{tikzcd}[column sep=2.2cm,row sep=2.4cm]
    {}_{\DotB}\gene{z}\bt \gene{t}_{\DotB}
    \arrow[bend left=12]{r}[above]{(S|S)}
    \arrow[in=-150,out=150,looseness=5.5]{rl}[above, very near start]{\substack{(S^2|D) \\ (D|D)}}
    \arrow[bend left=15, d, "\substack{(-|D) \\ (-|S^2)}" right]
    &
    {}_{\DotC} \gene{w} \bt \gene{u} _{\DotC}
    \arrow[bend left=12]{l}[below]{(S|S)}
    \arrow[in=-30,out=30,looseness=5.5]{rl}[above, near start]{\substack{(S^2|D) \\ (D|D)}}
    \arrow[bend right=15, d, "\substack{(-|D) \\ (-|S^2)}" left ]
    \\
    {}_{\DotB}\gene{z}\bt \gene{k}_{\DotB}
    \arrow[bend left=12]{r}[above]{(S|S)}
    \arrow[in=-150,out=150,looseness=5.5]{rl}[below, very near end]{\substack{(S^2|D) \\ (D|D)}}
    \arrow[bend left=15, u, "\substack{(S,S|1) \\ (D,D|D) \\ (S^2,D|D) \\ (D,S^2|D) \\ (S^2,S^2|D)}" left]
    &
    {}_{\DotC}  \gene{w} \bt \gene{v} _{\DotC}
    \arrow[bend left=12]{l}[below]{(S|S)}
    \arrow[in=-30,out=30,looseness=5.5]{rl}[below, near end]{\substack{(S^2|D) \\ (D|D)}}
    \arrow[bend right=15, u, "\substack{(S,S|1) \\ (D,D|D) \\ (S^2,D|D) \\ (D,S^2|D) \\ (S^2,S^2|D)}" right]
    \end{tikzcd}
    $$
    \caption{$ \Qbim \bt \Ybim $}
    \label{fig:bimodules:QY}
  \end{subfigure}
  \caption{Various bimodules appearing in Lemma~\ref{lem:main}}
  \label{fig:bimodules}
\end{figure}
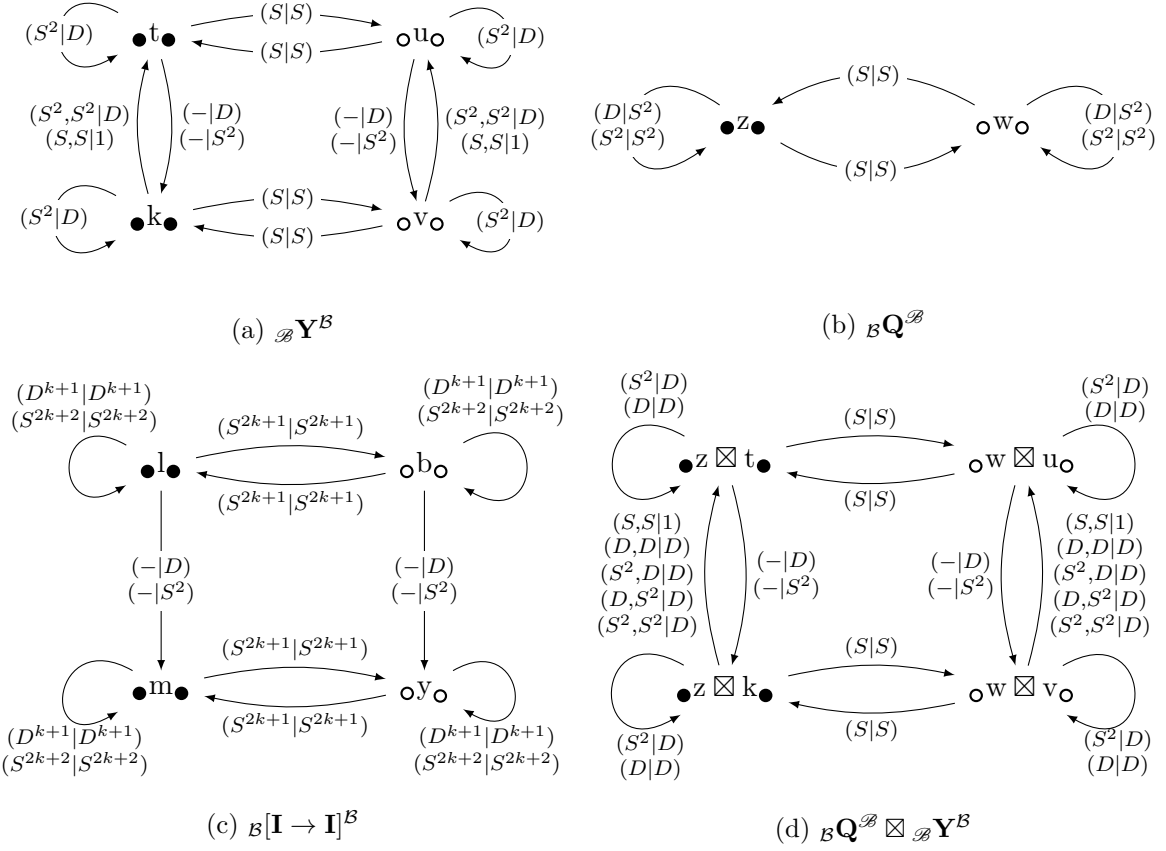
  \begin{figure}[p]
  \centering
  \begin{align*}
  f\co \mcIbim \longrightarrow \Qbim \bt \Ybim &\qquad
  \begin{tikzcd}[ampersand replacement = \&,column sep=1cm,row sep=1cm]
    \gene{l}
    \arrow[rrrrr, bend left=25, pos=0.3, dashed, "1" description]
    \&
    \gene{b}
    \arrow[rrrrr, bend left=25, pos=0.7, dashed, "1" description]
    \&\&\&\&
    \gene{z}\bt \gene{t}
    \&
    \gene{w}\bt \gene{u} 
    \\
    \gene{m} 
    \arrow[rrrrr, bend right=25, pos=0.3, dashed, "1" description]
    \arrow[rrrrru, bend left=20, pos=0.7, "\substack{(S^{2k+2}|S^{2k}) \\ (D^{k+2}|D^{k+1})}" description]
    \arrow[rrrrrru, bend right=20, "(S^{2k+3}|S^{2k+1})" description] 
    \&
    \gene{y}
    \arrow[rrrrr, bend right=25, pos=0.7, dashed, "1" description]
    \arrow[rrrrru,  bend right=8, near end, "\substack{(S^{2k+2}|S^{2k}) \\ (D^{k+2}|D^{k+1})}" description] 
    \arrow[rrrru,  "(S^{2k+3}|S^{2k+1})" description] 
    \&\&\&\&
    \gene{z}\bt \gene{k}
    \&
    \gene{w}\bt \gene{v}  
  \end{tikzcd}
  \\ 
  g\co \Qbim \bt \Ybim  \longrightarrow \mcIbim &\qquad
  \begin{tikzcd}[ampersand replacement = \&,column sep=1cm,row sep=1cm]
    \gene{l}
    \arrow[leftarrow,rrrrrrd, bend right=25, "(S^{2k+3}|S^{2k+1})" description] 
    \&
    \gene{b} 
    \&\&\&\&
    \gene{z}\bt \gene{t}
    \arrow[lllll, bend right=25, pos=0.7, dashed, "1" description]
    \&
    \gene{w}\bt \gene{u}
    \arrow[lllll, bend right=25, pos=0.3, dashed, "1" description]
    \\
    \gene{m}
    \&
    \gene{y}
    \&\&\&\&
    \gene{z}\bt \gene{k} 
    \arrow[lllll, bend left=25, pos=0.7, dashed, "1" description]
    \arrow[lllllu,  bend left=8, near end, "\substack{(S^{2k+2}|S^{2k}) \\ (D^{k+2}|D^{k+1})}" description]
    \arrow[llllu, "(S^{2k+3}|S^{2k+1})" description]
    \& 
    \gene{w}\bt \gene{v}  
    \arrow[lllll, bend left=25, pos=0.3, dashed, "1" description]
    \arrow[lllllu, bend right=20, pos=0.7,, "\substack{(S^{2k+2}|S^{2k}) \\ (D^{k+2}|D^{k+1})}" description]
  \end{tikzcd}  
  \end{align*}
  \caption{The isomorphisms between the type~AD bimodules $\Qbim \bt \Ybim$ and $\mcIbim$. Degree $k$ in actions means summation over non-negative integers.}
  \label{fig:AD_maps}
  \end{figure}
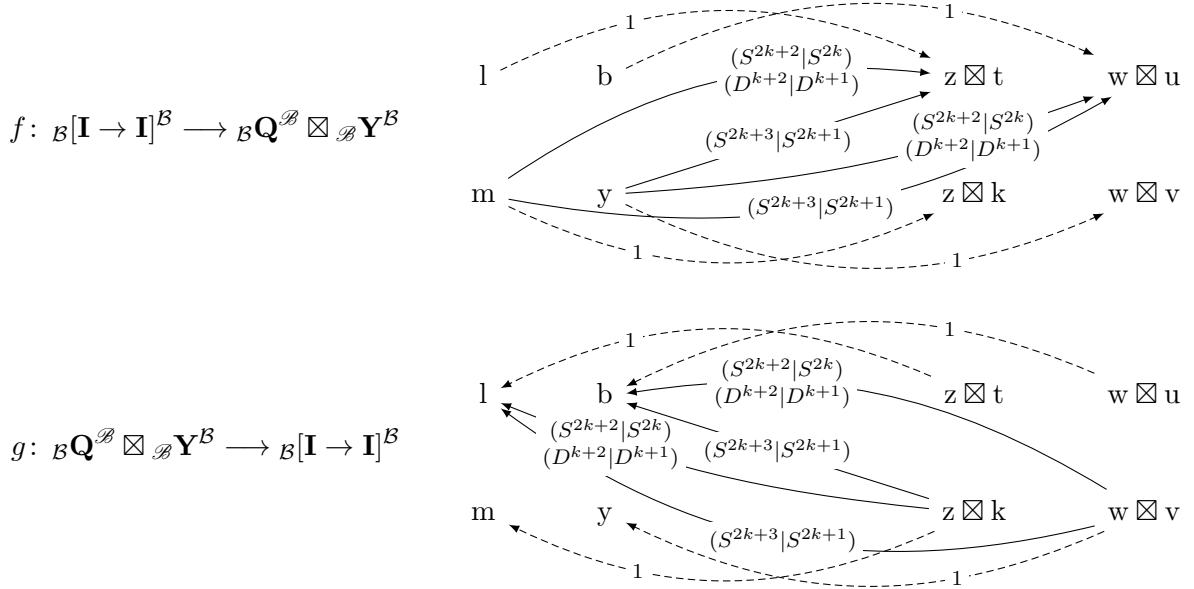

\begin{proof}
  We will prove a stronger statement, namely, that the bimodules $\mcIbim$ and $ \Qbim \bt \Ybim$ are chain isomorphic. The two bimodules are depicted in Figures~\ref{fig:bimodules:I} and~\ref{fig:bimodules:QY}, respectively. The following example explains our shorthand notation for the $\delta^1_{1+j}$-actions of type AD bimodules: 
  $$(\gene{z}\bt \gene{k}) \xrightarrow{(S^2,D| D)} (\gene{z}\bt \gene{t})  ~ \text{ represents an action } ~ S^2 \otimes D \otimes (\gene{z}\bt \gene{k}) \longrightarrow  (\gene{z}\bt \gene{t}) \otimes D ~ \text{ of the }~ \delta^1_{1+2}\text{-map.}$$ 
  Furthermore, $k$ means summation over non-negative integers: for example, the action $(D^{k+1}|D^{k+1})$ represents infinitely many actions, one for each $0\leq k \le \infty$.
  
  In Figure~\ref{fig:AD_maps}, we define two type~AD bimodule maps $f=f_1+f_2$ and $g=g_1+g_2$, where $f_1$ and $g_1$ denote the identity components of these maps indicated by the dashed arrows. 
  We claim that $f$ and $g$ are isomorphisms between the bimodules $\Qbim \bt \Ybim$ and $\mcIbim$. For this we need to check that:
  \begin{enumerate}
  \item $g_2\circ f_2 =0$, 
  $g_2 \circ f_1 =  g_1 \circ f_2$, 
  $g_1 \circ f_1 =\id$, so $(g_1+g_2)\circ(f_1+f_2)=\id_{\mcIbimSimple}$;
  \item $f_2 \circ g_2 =0$, 
  $f_1 \circ g_2  =  f_2 \circ g_1$, 
  $f_1 \circ g_1 =\id$, so $(f_1+f_2)\circ(g_1+g_2)=\id_{\QbimSimple \bt \YbimSimple}$;
  \item $f$ and $g$ are type~AD bimodule homomorphisms, ie $\partial(f)=\partial (g) =0$.
  \end{enumerate}
  We refer the reader to~\cite[Definition~2.2.43]{LOT-bim} for how the composition and the differentials of type~AD bimodule maps are defined.
  
  The first two statements follow in a straightforward way from the diagrams in Figure~\ref{fig:AD_maps}. Before we verify the third statement, observe that the bimodules and maps are symmetrical under switching $\gene{l}\leftrightarrow\gene{b}$,
  $\gene{m}\leftrightarrow\gene{y}$,
  $\gene{z}\boxtimes\gene{t}\leftrightarrow\gene{w}\boxtimes\gene{u}$ 
  and $\gene{z}\boxtimes\gene{k}\leftrightarrow\gene{w}\boxtimes\gene{v}$. So to check that $f$ is a homomorphism, it suffices to verify that the following components of $\partial (f)$ vanish:
  \begin{align*}
  &
  \gene{l}
  \rightarrow\gene{z}\boxtimes\gene{t}
  &&
   \gene{l}
  \rightarrow\gene{w}\boxtimes\gene{u}
  &&
  \underline{\gene{m}
  \rightarrow\gene{z}\boxtimes\gene{t}}
  &&
  \underline{\gene{m}
  \rightarrow\gene{w}\boxtimes\gene{u}}
  \\
  &
  \gene{l}
  \rightarrow\gene{z}\boxtimes\gene{k}
  &&
  \gene{l}
  \rightarrow\gene{w}\boxtimes\gene{v}
  &&
  \gene{m}
  \rightarrow\gene{z}\boxtimes\gene{k}
  &&
  \gene{m}
  \rightarrow\gene{w}\boxtimes\gene{v}
  \end{align*}
  This calculation is very simple in all but the two underlined cases. We consider those separately: the component $\gene{m}
  \rightarrow\gene{z}\boxtimes\gene{t}$ consists of the terms below, where we take the sum over all non-negative $l$ and $k$; the symbol $\partial^*$ represents the second term in the formula of $\partial(f)$ from~\cite[Figure~2]{LOT-bim}.
  \begin{align*}
  \partial^*(\gene{m}
  \rightarrow\gene{z}\boxtimes\gene{t})
  &
  =\partial^*(S^{2k+2}|S^{2k})+\partial^*(D^{k+2}|D^{k+1})
  \\
  \gene{m}
  \rightarrow\gene{z}\boxtimes\gene{t}\rightarrow\gene{z}\boxtimes\gene{t}
  &
  =(S^2,S^{2}|D)+(D,S^{2}|D)+(S^2,D^{k+2}|D^{k+2})+(D,D^{k+2}|D^{k+2})
  \\
  \gene{m}
  \rightarrow\gene{m}
  \rightarrow\gene{z}\boxtimes\gene{t}
  &
  =
  (S^{2},D^{l+1}|D^{l+1})+(D^{k+2},D^{l+1}|D^{l+k+2})+(S^{2k+2},S^{2l+2}|S^{2l+2k+2})
  \\
  \gene{m}
  \rightarrow\gene{z}\boxtimes\gene{k}
  \rightarrow\gene{z}\boxtimes\gene{t}
  &
  =(S,S|1)+(D,D|D)+(S^2,D|D)+(D,S^2|D)+(S^2,S^2|D)
  \\
  \gene{m}
  \rightarrow\gene{w}\boxtimes\gene{u}
  \rightarrow\gene{z}\boxtimes\gene{t}
  &
  =(S,S^{2k+3}|S^{2k+2})
  \\
  \gene{m}
  \rightarrow\gene{y}
  \rightarrow\gene{z}\boxtimes\gene{t}
  &
  =(S^{2k+3},S^{2l+1}|S^{2l+2k+2})
  \end{align*}
  All these contributions cancel. Similarly, the component $\gene{m}
  \rightarrow\gene{w}\boxtimes\gene{u}$ consists of the following terms which also cancel each other:
  \begin{align*}
  \partial^*(\gene{m}
  \rightarrow\gene{w}\boxtimes\gene{u})
  &
  =\partial^*(S^{2k+3}|S^{2k+1})
  \\
  \gene{m}
  \rightarrow\gene{m}
  \rightarrow\gene{w}\boxtimes\gene{u}
  &
  =(S^{2k+3},S^{2l+2}|S^{2l+2k+3})
  \\
  \gene{m}
  \rightarrow\gene{z}\boxtimes\gene{t}
  \rightarrow\gene{w}\boxtimes\gene{u}
  &
  =(S,S^{2k+2}|S^{2k+1})
  \\
  \gene{m}
  \rightarrow\gene{y}
  \rightarrow\gene{w}\boxtimes\gene{u}
  &
  =(S^{2k+2},S^{2l+1}|S^{2l+2k+1})
  \end{align*}
  This finishes the proof that $f$ is a homomorphism. To see that $g$ is a homomorphism, observe that this map is symmetric to $f$ under reversing the direction of all arrows and switching $\gene{l}\leftrightarrow\gene{m}$,
  $\gene{b}\leftrightarrow\gene{y}$,
  $\gene{z}\boxtimes\gene{t}\leftrightarrow\gene{z}\boxtimes\gene{k}$ and 
  $\gene{w}\boxtimes\gene{u}\leftrightarrow\gene{w}\boxtimes\gene{v}$.
  \end{proof}

\begin{corollary}
  The square labelled \(\boxdot\) in Figure~\ref{fig:diagram:Algebras:overview} commutes up to homotopy, and so does the square labelled \(\boxdot\) in Figure~\ref{fig:diagram:Mod:overview}. 
\end{corollary}
\begin{proof}
  Observe that the chain isomorphisms \(f\) and \(g\) factor through \(\FigureEightSubAlg\). 
\end{proof}

\begin{proof}[Proof of the Main Theorem]
Consider Figure~\ref{fig:diagram:Mod:overview}. The commutativity of this diagram, which we have established in this and the previous two sections, implies that for any pointed 4-ended tangle $T$, the image of $L_T$ under the quasi-isomorphism $\Fmap $ is equal to  $\DD(T)^{\FigureEightAlg}$ up to chain homotopy. Thus the invariants $L_T$ and $\DD(T)^{\FigureEightAlg}$ are equivalent. Moreover, $\Mod^{\FigureEightAlg}$ is a full subcategory of $\Mod^{\BNAlgH}$ which implies that any two objects in $\Mod^{\FigureEightAlg}$ which are chain homotopic as objects in $\Mod^{\BNAlgH}$ are also chain homotopic in $\Mod^{\FigureEightAlg}$. Thus $\DD_1(T)$ is equivalent to $\DD(T)^{\FigureEightAlg}$ and hence to $L_T$. 
\end{proof}

By passing from the category $\Mod^{\FigureEightSubAlg}$ to the category $\Mod^{\FigureEightAlg}$ (or equivalently from $\Mod^{\HHHKSubAlg}$ to $\Mod^{\HHHKAlg}$) there may be a loss of information since there are potentially more chain homotopies in the latter than in the former. Thus we end with:

\begin{question} 
Is $\DD(T)^{\FigureEightSubAlg}$ (equivalently $\DD(T)^{\HHHKSubAlg}$) a stronger tangle invariant than $\DD_1(T)$ (equivalently $L_T$)?
\end{question}

\newcommand*{\arxivPreprint}[1]{ArXiv preprint \href{http://arxiv.org/abs/#1}{#1}}
\newcommand*{\arxiv}[1]{(ArXiv:\ \href{http://arxiv.org/abs/#1}{#1})}
\bibliographystyle{alpha}
\bibliography{comparison}
\end{document}